\documentclass[12pt,oneside]{amsart}
\usepackage{amssymb,amsmath,amsthm}
\usepackage{graphicx}
\usepackage{color}
\def\e{{\rm e}}
\def\eps{\varepsilon}

\def\d{{\rm d}}

\def\dim{{\rm dim}}

\def\ddt{\frac{\d}{\d t}}

\def\AA {{\mathfrak A}}

\def\R {\mathbb{R}}

\def\V {{\mathcal V}}

\def\B {{\mathcal B}}
\def\C {{\mathcal C}}
\def\D {{\rm dom}}
\def\E {{\mathcal E}}
\def\A {{\mathcal A}}
\def\AA {{\mathbb A}}
\def\K {{\mathcal K}}
\def\J {{\mathcal J}}
\def\I {{\mathcal I}}

\def\F {{\mathcal F}}

\def\Q {{\mathcal Q}}

\def \l {\langle}
\def \r {\rangle}
\def \pt {\partial_t}
\def \ptt {\partial_{tt}}

\def \and{\quad\text{and}\quad}

\newcommand{\cic}[1]{\mbox{\boldmath$#1$}}

\def \au {\rm}

\def \jou {\rm}

\def \no#1#2#3 {{\bf #1} (#3), #2.}
\def \eds#1#2#3 {#1, #2, #3.}

\newtheorem{proposition}{Proposition}[section]
\newtheorem{theorem}{Theorem}[section]
\newtheorem{corollary}{Corollary}[section]
\newtheorem{lemma}{Lemma}[section]
\theoremstyle{definition}

\newtheorem{remark}{Remark}[section]
\newtheorem*{remark*}{Remark}
\newtheorem*{warn*}{A word of warning}

\numberwithin{equation}{section}

\title[3-dimensional oscillon equation]
{The 3-dimensional oscillon equation}

\author[F.\ Di Plinio]
{Francesco Di Plinio}
\address{
Institute for Scientific Computing and Applied Mathematics
\newline\indent
Indiana University
\newline\indent
Bloomington, IN 47405 - USA}
\email{fradipli@indiana.edu {\rm (F.\ Di Plinio)} }
\email{temam@indiana.edu {\rm (R.\ Temam)} }

\author[G.\ S.\ Duane]
{Gregory S.\ Duane}

\address{
Rosenstiel School of Marine and Atmospheric Sciences
\newline\indent
University of Miami
\newline\indent
Miami, FL 33149
\newline\indent \vskip-3mm
Dept. of Atmospheric and Oceanic Sciences
\newline\indent
University of Colorado
\newline\indent
Boulder, CO 80309
}
\email{gregory.duane@colorado.edu {\rm (G.\ S.\ Duane)} }

\author[R.\ Temam]
{Roger Temam}

\thanks{This article appeared in print on Boll. Unione Mat. Ital. Ser. IX 5 (2012), no. 1, 19-54, MathScinet record at  http://www.ams.org/mathscinet-getitem?mr=2919647.
   Work partially
supported by the National Science Foundation under the grants
NSF-DMS-0604235, NSF-DMS-0906440, and by the Research Fund of Indiana University.}
\subjclass{Primary: 37L30, 35B41; Secondary: 83D05.}
 \keywords{Oscillon equation, nonautonomous attractors, fractal dimension.}

\dedicatory{Dedicated to the memory of Giovanni Prodi}

\begin{document}

\begin{abstract}On a bounded smooth domain $\Omega \subset \R^3$, we consider the generalized 
oscillon equation  $$\ptt u(x,t) +\omega(t) \pt u(x,t) -\mu(t)\Delta u(x,t) +
V'(u(x,t)) =0, \qquad x\in \Omega \subset \R^3, t \in \R,
$$
with Dirichlet boundary conditions, where $\omega$ is a time-dependent damping, $\mu$ is a time-dependent squared speed of propagation, and  $V$ is a nonlinear potential of critical growth.

Under structural assumptions on $\omega$ and $\mu$
we establish the existence  of a pullback global attractor
$\A=\A(t)$ in the sense of \cite{DDT}. Under additional assumptions on  $\mu$, which include the relevant physical cases, we obtain optimal regularity of the pullback global attractor and finite-dimensionality of the kernel sections.
\end{abstract}

\maketitle

\section{Introduction}
Let $\Omega \subset \R^3$ be a bounded domain with   smooth boundary. We consider the generalized 
oscillon equation  
\begin{equation}\ptt u(x,t) +\omega(t) \pt u(x,t) -\mu(t)\Delta u(x,t) +
V'(u(x,t)) =0, \qquad x\in \Omega \subset \R^3, t \in \R, \label{SYS-INTRO}
\end{equation}
with Dirichlet boundary conditions, where $\omega$ is a time-dependent damping, $\mu$ is a time-dependent squared speed of propagation, and  $V$ is a nonlinear potential. This equation, as well as its simpler version in dimension one studied in \cite{DDT},  has been proposed to described some long-lived structures (termed \emph{oscillons}) which appear in the formation of the universe; see e.g. \cite{Farhi05,FGGIRS},   for more details on the physical context, as well as \cite{DDT}. Further studies on the oscillon equation and on the physical context can be found in \cite{H}, and  \cite{Forgacs,FFM}, which are part of a series of articles on the subject.

 The concept of {\it pullback attractor} has been shown to capture an enlarged notion of
dissipativity that is applicable to Hamiltonian systems in which phase-space volume is conserved.
In previous work \cite{DDT}, it was suggested that the long-lived coherent
structures, {\it oscillons}, in a dynamical system describing a scalar field in an
expanding universe might naturally be described in terms of a non-trivial pullback attractor. Here
we show that the construct is structurally stable: the existence of a pullback attractor is established for a large class
of expansion scenarios. A pullback attractor is also shown to exist in three-dimensional models, consistently
with work on the occurrence of oscillons in three-dimensional particle physics models \cite{Farhi05}.
Gauge fields are probably required for three-dimensional oscillons to be long-lived, but metastable oscillons
have been studied in three-dimensional scalar field models \cite{Forgacs}, resembling the one used here.
 
The article is organized as follows. After this introduction, describing the background and motivations,  Section 2 summarizes the time-dependent attractor framework developed in \cite{DDT}.  The results of \cite{DDT} are supplemented with a new result (Corollary \ref{cor:uniqueness}) which establishes an  important uniqueness property of the pullback attractor. That is, if the existence Theorem \ref{existence} applies, then the resulting  pullback-bounded attractor is necessarily the unique pullback-bounded pullback attractor.

Section 3 contains the abstract formulation of the evolution problem associated with \eqref{SYS-INTRO} in the setting of time-dependent spaces, as well as the assumptions on the time-dependent damping  term $\omega$ and squared speed of propagation $\mu$. Regarding the nonlinear potential $V$, we require it to be of dissipative nature, and to have   polynomial growth $q $ at most $4$; the growth rate $q=4$ is critical for the well posedness of the problem (as well as for an autonomous damped wave equation in space dimension three) in the weak sense.  In the physical model from relativistic mechanics, $\mu$ is usually taken to be a decreasing function on $\R$, unbounded for $t \to -\infty $ and vanishing at $+\infty$; however, we are able to deal with a more general class of time-dependent $\mu$'s, not necessarily decreasing,  which are of interest for other physical models.  (for instance, wave propagation in media with time-dependent shape). Namely, we merely require that for each time $t$, the growth of $\mu$ on $(-\infty,t]$ is at most exponential, with rate comparable to the damping coefficient $  \omega(t)$. Under our assumptions, the damping term $\omega(t)$ is allowed to (possibly) vanish at $+\infty$; this ensures that the physical model of the reheating phase of inflation (see \cite{FGGIRS}) falls into the scope of our analysis.

In Section 4, we list, comment and motivate the main results of the paper. In Theorem \ref{wp:th}, we show that the evolution problem associated with \eqref{SYS-INTRO} generates a strongly continuous process $z \mapsto S(t,s)z$, depending continuously on the initial data $z$, and that the process $S(t,s)$ is of dissipative nature, i.e.\ possesses a pullback-bounded absorber.
With Theorem \ref{osc:attractor}, we establish the existence of a pullback attractor $\A=\{\A(t):t \in \R\}$ for the process $S(t,s)$. Regularity properties of the pullback attractor, namely, boundedness of the kernel sections $\A(t)$ in a more regular space, are addressed in Theorem \ref{thm:reg}. In order to obtain the regularity result, we exploit an additional integrability property of the time-derivatives of the solution; for this property to hold in our time-dependent setting, a further local $L^p$-integrability condition on $\mu'$, condition \eqref{pospart}, is needed. However, the scope of condition \eqref{pospart} includes a wide range of qualitative behaviors for $\mu$: for example, \eqref{pospart} holds whenever $\mu$ is a decreasing function, or more generally, has finitely many critical points; oscillatory behavior for arbitrarily large negative times is also allowed, under additional assumptions (see Remark \ref{rem:M2} for details).   Finally, we show in Theorem \ref{thm:fd} that the kernel sections $\A(t)$ have finite fractal dimension.\\ 
Sections 5 to 8 contain the proofs of the main results. 
\vskip2mm

This article is dedicated with much consideration to the memory of Giovanni Prodi, who did so much for the theory of partial differential equations, and especially for the Navier-Stokes equations. In particular, we mention the articles \cite{FP,LP, P} quoted so many times by the third author (RT).

\section{Attractors in Time-Dependent Spaces} \label{sect:attractors}
As anticipated in the Introduction, in this section we summarize the definitions and main results concerning attractors in time-dependent spaces. For the interested reader, complete proofs of the theorems and corollaries listed below, as well as comparison with the preexisting literature and appropriate references,  can be found in \cite{DDT}, where the  framework of time-dependent spaces has been introduced for the first time.

\subsection*{Process.} For $t \in \R$, let $X_t$ be a family of Banach spaces endowed with norms $\|\cdot\|_{X_t}$ (see  \eqref{spaces} for an example). A (continuous) \emph{process}  is a two-parameter family of
mappings $\{S(t,s): X_s \to X_t\}_{s \leq t}$ with properties
\begin{itemize}
  \item[(i)] $S(t,t) = \textrm{Id}_{X_t}$;
  \item[(ii)] $S(t,s) \in \C(X_s, X_t)$;
   \item[(iii)] $S(\tau,t)S(t,s) = S(\tau,s)$ for $s \leq t \leq \tau.$
\end{itemize}

\subsection*{Pullback-bounded family.} A family of subsets $\B=\{\B(t) \subset X_t\}_{t \in \R}$
is \emph{pullback-bounded} if \footnote{Here, for $D$ subset of a
Banach space $X$, $\displaystyle \|D\|_{X} = \sup_{z \in D}
\|z\|_{X}$.}
$$
R(t)=\sup_{s\in (-\infty, t]}\|\B(s) \|_{X_s} < \infty \qquad \forall t \in \R,
$$
{i.e. the function $s \mapsto \|\B(s) \|_{X_s}$ is bounded on  $s \in (-\infty,t]$ for each $t \in \R$.
}
\vskip2mm

\subsection*{Pullback absorber.} A  pullback-bounded family  $\AA=\{\AA(t)\}$ is called \emph{pullback absorber} if for every pullback-bounded family $\B$ and for every $t \in \R$ there exists $t_0=t_0(t)\leq t$ such that
$$
S(t,s) \B(s) \subset \AA(t), \qquad \forall s \leq t_0.
$$
\vskip2mm

\subsection*{Time-dependent $\omega$-limit.} Given a  family of sets $\B$, its time-dependent $\omega$-limit is the family
 $\omega_\B=\{\omega_\B(t) \subset X_t\}_{t \in \R}$, where $\omega_\B(t)$ is defined as
$$
\omega_\B(t)=\bigcap_{\tau \leq t} \overline{\bigcup_{s \leq \tau} S(t,s) \B(s)},
$$
and the above closures are  taken in $X_t$. A more concrete characterization is the following:
$$
\omega_\B(t) = \{z \in X_t : \exists\, s_n \to - \infty, z_n \in \B(s_n) \textrm{ with } \|S(t,s_n)z_n -z \|_{X_t} \to 0 \textrm{ as } n \to \infty\}.
$$

  \subsection*{Time-dependent global attractor.} A  family of compact subsets  $\A=\{\A(t) \subset X_t\}_{t \in \R}$ is called  \label{def:attractor}\emph{time-dependent global attractor} for the process $\{S(t,s)\}_{s\leq t}$ if it fulfills the following  properties:
\begin{itemize}
  \item[(i)](invariance) $S(t,s) \A(s) = \A(t)$, for every $s \leq t$ ;
  \item[(ii)] (pullback attraction) for every pullback-bounded family $\B$ and every $t \in \R$, \footnote{
For a Banach space $X$ and $A,B \subset X$, the Hausdorff semidistance is defined as
$$
\textrm{dist}_{X} (A,B) = \sup_{x \in A} \inf_ {y \in B} \|y-x\|_X.
$$
From the definition, $\textrm{dist}_{X} (A,B) =0$ if and only if $A$ is contained in the closure of $B$.
}   $$ \displaystyle \lim_{s \to - \infty } \textrm{dist}_{X_t} (S(t,s)\B(s),\A(t)) =0.$$
\end{itemize}
If property (ii) holds uniformly with respect to $t \in \R$, $\A$ is called a \emph{uniform} time-dependent global attractor.
\begin{remark} \label{uniqueness} In general, conditions (i)-(ii) are not sufficient to guarantee the  uniqueness of the time-dependent attractor.  For example, consider the simple ODE, $y'+y=0$, and denote by $S(\cdot,\cdot)$  the process it generates  on $\R$, i.e.\ $S(t,s) x = x\e^{-(t-s)}$. The process $S(\cdot,\cdot)$ has infinitely many time-dependent attractors in the sense of the definition above; they are given by $\A_c=\{\A_c(t)=c\e^{-t}\}, c\in \R$. \emph{However, only $\A_0$ is also a pullback-bounded family.}

\noindent 
Indeed, if we require in addition
\begin{itemize}
  \item[(iii)] $\A$ is a pullback-bounded family,
\end{itemize}
then there exists at most one family satisfying (i)-(iii), i.e.\ \textit{a
pullback-bounded time-dependent global attractor is unique in the
class of pullback-bounded families. } 
\end{remark}

\begin{remark} \label{uniquenessprime} Note that the definition does not require the time-dependent attractor $\{\A(t)\}_{t \in \R}$ to be pullback-bounded. However, the time-dependent attractor we are going to construct (see Theorem \ref{existence} and Corollary \ref{cor:uniqueness}) will indeed be pullback-bounded, and thus unique in the sense of Remark \ref{uniqueness}.\end{remark}
\subsection*{Existence of the global attractor}
The shorthand $\alpha_t$ stands for the Kuratowski measure \footnote{If $X$ is a Banach space, the \emph{Kuratowski measure of noncompactness} of a subset $A \subset X$ is defined by
$$
\alpha(A) = \inf \{\delta>0: A \textrm{ is covered by finitely many $X$-balls of radius }\delta \}.
$$ } in the space $X_t$. We remark that, for fixed $s,t\in \R$, $\alpha_s$ and $\alpha_t$ are equivalent measures of noncompactness whenever  there is a Banach space isomorphism between $X_s$ and $ X_t$.

\begin{theorem} \label{existence}
Assume that the process $S(\cdot,\cdot)$ possesses an absorber $\AA$ for which
\begin{equation} \label{alphaa}
\lim_{s \to  -\infty} \alpha_t(S(t,s)\AA(s)) =0, \qquad \forall t \in \R.
\end{equation}
Then, $\omega_\AA$
is a global attractor for $S(\cdot,\cdot).$
\end{theorem}
 We follow up the theorem with some important corollaries. In the first, we show that the construction of Theorem \ref{existence} always yields the unique   pullback-bounded global attractor.
\begin{corollary}\label{cor:uniqueness}
Under the assumptions of Theorem \ref{existence}, 
$$
\A(t) =\omega_\AA (t) \subset 	 \AA(t), \qquad \forall t \in \R;
$$ In particular, $\A$ is a   pullback-bounded family, and therefore
unique  in the sense of Remark \ref{uniqueness}.
\end{corollary}
\begin{proof}
Let $t \in \R$, $z \in \A(t)$ be fixed. By definition of an $\omega$-limit family, there exist $s_n \to -\infty, z_n \in \AA(s_n)$ such that $\|
z - S(t,s_n) z_n
\|_{X_t} \to 0$ as $n \to \infty$. But $\AA$ is a pullback-bounded family, and therefore absorbs itself: $S(t,s)\AA(s)\subset \AA(t)$ for every $s \geq s_\star(t)$. Hence $S(t,s_n)z_n \in \AA(t)$ for  $n$ large enough, so that {$z  \in \overline{\AA(t)}$. This, in particular, implies that the family $\{\A(t)\}$ is pullback-bounded and thus absorbed by $\AA$, i.e $S(t,s(t))\A(t) \subset \AA(t)$ for some $s(t) \leq t$. The stronger inclusion $\A(t) \subset 	 \AA(t)$ then follows again by the invariance of $\A$.}
\end{proof}

The second corollary is a concrete reformulation of
Theorem~\ref{existence}, proven in \cite{DDT}, and completed with the uniqueness result of Corollary \ref{cor:uniqueness}.
\begin{corollary}\label{cor:attractor}If the process $S(\cdot,\cdot)$ with absorber $\AA$ possesses a decomposition
$$
S(t,s) \AA(s) = \mathfrak{P}(t,s) + \mathfrak{N}(t,s)
$$
where
$$
\lim_{s \to -\infty } \|\mathfrak{P}(t,s) \|_{X_t} =0, \qquad \forall t \in \R,
$$
and $\mathfrak{N}(t,s)$ is a compact subset of $X_t$ for all $t \in \R$ and $s \leq t$,
then  $\A(t)=\omega_\AA(t)$ is the unique (in the sense of Remark \ref{uniqueness}) global attractor for the process $S(\cdot, \cdot)$ .
\end{corollary}

Finally, we dwell on further regularity properties of the pullback global attractor.
\begin{corollary}\label{cor:regularity}
Let $Y_t$ be a further family of Banach spaces satisfying, for
every $t \in \R$,\footnote{With $Y\Subset X$ we indicate compact
injection of the Banach space $Y$ into the Banach space $X$.}
\begin{itemize}
\item[$\cdot$] $Y_t \Subset X_t$;
\item[$\cdot$]  closed balls of $Y_t$ are closed$\;$\footnote{For example, this holds when $Y_t$ is reflexive and compactly embedded into $X_t$.} in $X_t$ .
\end{itemize}
Under the same assumptions as in  Corollary~\ref{cor:attractor},
if in addition
$$
\sup_{s \in (-\infty,t]} \|N(t,s)\|_{Y_t} = h(t) < \infty \qquad \forall t \in \R,
$$
then the global attractor  satisfies
$$\|\A(t)\|_{Y_t} \leq h(t)\qquad \forall t \in \R.$$
\end{corollary}

\section{The 3D oscillon equation with a general potential}
In this section, we   state the main assumptions and then  cast  the evolution problem associated with \eqref{SYS-INTRO} in the abstract framework of processes in time-dependent spaces as described in Section 2.

 \subsection*{Notation}
Let $\Omega$ be a smooth bounded domain in $\R^3$.
In the following, $|\cdot|$ and $\l \cdot,\cdot\r$ denote
respectively the standard norm and scalar product on $L^2(\Omega)$;
$A$ denotes  $-\Delta$ on $\Omega$ with Dirichlet boundary
conditions, with domain ${\D}(A)=H^2(\Omega)\cap H_0^1(\Omega)$.
For $\ell\in\R$, we define the scale of Hilbert spaces
$H^\ell={\D}(A^{\ell/2})$, endowed with the standard
inner product and norm $$\l u,v\r_{\ell}=\l A^{\ell/2}u,A^{\ell/2}v\r,\qquad |u|_\ell=|A^{\ell/2}u|.$$  
The symbols $c$ and $\Q$ will stand respectively for a generic
positive constant and a generic positive increasing continuous
function; both  may be different in different occurrences. When an
index is added, (e.g.\ $c_0, \Q_0$), the positive constant (resp.\
function) is meant to be specific and will be referred to subsequently. Similarly, the symbols
$\Lambda,\Lambda_\imath$ will denote certain energy-like
functionals occurring in the proofs.  
\subsection{Definition of the problem and assumptions on the nonlinearity}  \label{Sect3subnl}
 \vskip2mm
 We study the \emph{oscillon
equation} in space dimension $n=3$ with Dirichlet boundary
conditions
\begin{equation}
\label{SYS} \tag{P}
\begin{cases}
\displaystyle
\ptt u(t)+\omega(t) \pt u(t) + \mu(t) A u(t) + \varphi(u(t))=0, & t \geq s,\\ \\
u(s)= u_0 \in H^1, \pt u(s) = v_0  \in H.
\end{cases}
\end{equation}
We consider a (nonlinear) potential $V \in \C^3(\R)$,
such that $V(0)=0$, and $\varphi=V'$  satisfies the following
assumptions:
\begin{itemize}
\item[(H0)] $\varphi(0)=0$;
\item[(H1)] there exist $a_0,a_2> 0$, $a_1,a_3\geq0$, $ q \in [2,4]$ such that
$$ a_0|y|^{q-2} -a_1 \leq \varphi'(y) \leq a_2 |y|^{q-2}+a_3.
$$
When $q=2$ (sublinear case), we assume $a_0 > a_1$.\end{itemize}

The case $q=4$ is 
\emph{critical} for well-posedness of \eqref{SYS} (as well as for an autonomous damped wave equation in space dimension three) in the weak sense.\footnote{By this, we mean that for nonlinearities $V$ growing faster than a polynomial of order $4$ the uniqueness  of \emph{weak solutions} to \eqref{SYS}, i.e. $$ u \in C\big([s, T], H^1_0(\Omega)\big), \; \pt u \in L^2\big([s,T],L^2(\Omega)\big), \; \ptt u \in L^2\big([s,T],H^{-1}(\Omega)\big), \qquad \forall T \geq s,$$ is not guaranteed.
}

Since $V(0)=\varphi(0)=0$, two consecutive integrations
of (H1) yield
\begin{equation} \label{bound:V1}
\textstyle  \frac{a_0}{q(q-1)}|y|^{q} -\frac{a_1}{2} y^2 \leq V(y) \leq \textstyle  \frac{a_2}{q(q-1)}|y|^{q} +\frac{a_3}{2} y^2.
\end{equation}
Moreover, integrating by parts and using (H0)-(H1), we have 
\begin{equation} \label{fritto} y\varphi(y) \geq V(y) +\frac{a_0}{q} |y|^q - \frac{a_1}{2}y^2 \geq V(y) - c_0,
\end{equation}
for some $c_0\geq 0$ depending only on $a_0,a_1,q$. In particular,
we can take $c_0=0$ whenever $a_1=0$.
 We set also $$ \V(u) =\int_\Omega V(u(x))\, \d x. $$ In view of \eqref{bound:V1},
$\V(u) $ is well defined for every $u \in L^q(0,1)$, and \begin{equation}
\label{bound:V} b_0(\| u\|_{L^{q}}^{q}+|u|^2)-b_1 \leq \V (u) \leq
b_2(\| u\|_{L^{q}}^{q}+|u|^2),
\end{equation}
with $b_0,b_2>0$ and $b_1\geq 0$ depending only on the
$a_\imath$ ($\imath=0,\ldots,3$) and $q$;
in particular, $b_1=0$ whenever $a_1=0$.
\begin{remark}
We point out that the potential corresponding to 
$
\varphi(y) = y^3-y, 
$
(i.e.\ the well-known  $\phi^4$ extension of Klein-Gordon theory) falls into the scope of our assumptions (H0)-(H1), and of assumption (H2), which will be stated below in Section 4.
\end{remark}

\subsection{Assumptions on the time-dependent terms}  \label{Sect3subtd} We now specify the hypotheses we make on $\omega$ and $\mu$. See Remark \ref{physcond} below for the specific form of $\omega(t)$ and $\mu(t)$ in the case of an expanding universe. 

\subsubsection*{Assumptions on $\omega$}The damping coefficient $\omega: \R \to \R^+$ is assumed to be a decreasing strictly positive  differentiable function, with $\omega(t)$   bounded as $t \to -\infty$ (and thus on all of $\R$), and we set
$$
W:=\sup_{t \in \R} \omega(t)  < \infty.
$$
 Observe that the degeneracy $\lim_{t \to +\infty} \omega(t)=0$ is allowed. We associate with $\omega$ the \emph{decay rate} $\eps_\omega: \R \to \R^+$, defined as the function
\begin{equation} \label{epsom}
\eps_\omega(t) = \textstyle \frac{1}{16} \min\left\{1,\omega(t),\frac{c_1}{(1+W)}  \right\},  
\end{equation}
where $c_1>0$ is a positive constant depending only on $V$ as specified later. 
\subsubsection*{Assumptions on $\mu$} The main structural assumption on $\mu$ is as follows: $\mu(t)>0$ for all $t \in \R$, and there exists a   function $\alpha: \R\to[0,\infty) $, such that
\begin{equation} \label{assmu} \tag{M1}
   \mu'(t) \leq 2\alpha( t) \mu(t), \quad\textrm{with} \quad \sup_{s \leq t} \alpha(s) \leq \eps_\omega(t), \qquad \forall t \in \R.
\end{equation}
See Remarks  \ref{murmk1}, \ref{murmk1.5},  \ref{physcond} and \ref{rmkmu2} below for examples.
\begin{remark}\label{murmk1}
Any positive decreasing function $\mu$ satisfies \eqref{assmu} with $\alpha=0$, independently of how the positive function $\omega$ is chosen.
\end{remark}
\begin{remark} \label{murmk1.5} 
In the   case of constant damping $\omega(t)\equiv \omega >0$, $\eps_{\omega}$ is independent of $t$, and with $\alpha=\eps_\omega$ assumption \eqref{assmu} is equivalent to
\begin{equation} \label{assmu2}
\mu(t_2)  \leq \mu(t_1) \e^{2\eps_\omega (t_2-t_1)}, \qquad \forall t_1 \in \R, t_2 \geq t_1;
\end{equation}
that is,    $\e^{-2\eps_{\omega} t}\mu(t)$ is a decreasing function on $\R$. 
The case of constant damping $\omega$ and not necessarily decreasing $\mu$ is relevant in the study of the autonomous damped wave equation in a time-dependent domain, for instance, $\Omega_t=[0,a(t)]^3$. A rescaling produces the nonautonomous problem \eqref{SYS} on the \emph{fixed} domain $\Omega=[0,1]^3$,  with $\mu=\frac{1}{a^2}$. \end{remark}

\begin{remark} \label{physcond}We explain how   the time-dependency in \eqref{SYS} described by E.\ Fahri et al.\ in \cite{FGGIRS} for an expanding universe fits into our framework. If $a=a(t)$ denotes the rate of expansion of the universe, the physical model prescribes
$$
\mu(t) = \frac{1}{a(t)^2}, \qquad \omega(t)= \frac{a'(t)}{a(t)} = -\frac{1}{2} \ddt \left( \frac{1}{a(t)^2} \right) a(t)^2 = -\frac{1}{2} \frac{\mu'(t)}{\mu(t)}.
$$
Therefore, to ensure that the damping $\omega$ is a strictly positive function,  we have to require $\mu$  to be a decreasing function, which is the same as requiring the rate of expansion $a=a(t)$ to be a strictly increasing function, in agreement with the idea of an expanding universe. Hence, (M1) holds true with $\alpha\equiv 0$.

Moreover, we have to require that $\omega$ is decreasing, i.e., referring to the above form of $\omega$,
$$
\frac{\mu'(t)}{\mu(t)} \geq    \frac{\mu'(s)}{\mu(s)}, \qquad \forall t \geq s.
$$ If we set $\varpi(t)=\log(\mu(t))$, this can be rewritten as 
$
\varpi'(t) \leq \varpi'(s)
$,
for each $t \geq s$, i.e. $\varpi'$ is an increasing function, i.e. $\varpi=\log \mu$ is a convex function. But this is the same as saying that $a=\e^{-\varpi/2}$ is logarithmically concave (not necessarily logarithmically \emph{strictly} concave).
 Summarizing,  the assumptions  on the expansion rate $a(t)$ that are  needed to fit the described  physical case into our analysis are
\begin{align} \label{graham-hyp}
 a'(t) > 0 \; \textrm{ a.e. } t \in \R \qquad  &\big( \equiv \mu'(t) < 0 \; \textrm{ a.e. } t \in \R \big) \\
 \log a \textrm{ is a concave function},  \qquad&  \big( \equiv \log \mu \textrm{ is a convex function} \big) 
\end{align}
The article \cite{DDT} focuses on the most common case, where $a'/a=H>0$, so that  $a(t)=\e^{Ht},$ and $\log a(t)=Ht$ is concave (not strictly concave, but this is enough). Hence $a(t)$ is increasing and logarithmically concave, and the assumptions are satisfied.
Moreover, in \cite{FGGIRS} the authors cite as interesting the case of  (rapidly) decreasing $\omega$ (reheating phase of inflation), which fits the assumptions above as well ($a$ will be a logarithmically strictly concave function in that case).
For instance, the following is an example of a strictly increasing logarithmically concave function which is furthermore logarithmically \emph{strictly} concave for $t>0$:
$$
a(t)= \begin{cases} \e^{t+2} &  t\leq 0 \\ {\exp({2\sqrt{t+1}})} &  t >0 \end{cases} \leadsto a'(t)= \begin{cases} \e^{t+2} &  t\leq 0 \\ \frac{\exp({2\sqrt{t+1}})}{\sqrt{t+1}} & t >0 \end{cases} \leadsto \omega(t)= \begin{cases} 1  &  t\leq 0 \\ \frac{1}{\sqrt{t+1}} & t >0. \end{cases}
$$
\end{remark}

\begin{remark} \label{rmkmu2}
Regarding our assumptions on the damping $\omega$, which we recall is required  to be positive decreasing and bounded at $-\infty$, two significant examples are \begin{itemize}
  \item[$\cdot$] constant damping: $\omega(t)\equiv W >0$;
  \item[$\cdot$]  damping vanishing at $+\infty$: e.g $\omega_{\textrm{van}}(t)= \frac{W}{1+\e^{t}}.$ 
\end{itemize}
\noindent
Observe that if the damping is of the form $\omega_{\textrm{van}}$, a sufficient condition for our assumption  (M1) to hold is  
$$
\mu'(t) \leq c \min\{1,\e^{-t}\}\mu(t), \qquad \forall t \in \R,
$$
where $c $ is a constant that can be explicitly computed and depends on $W$ and $c_1$. Indeed, a suitable choice of the function $\alpha$ in this case is given by $\alpha(t)=c\min\{1,\e^{-t}\}$, with $c>0$ small enough.
\end{remark}

\subsection{The functional setting} \label{Sect3subfs}
We rewrite Problem \eqref{SYS} in our abstract framework. For $t,\ell \in \R$, we introduce the Banach spaces
\begin{equation} \label{spaces}
X_t^\ell = H^{1+\ell} \times H^\ell \quad \textrm{with norms} \quad \|(u,v) \|_{X_t^\ell}= \mu(t)^{1/2}|u |_{1+\ell} + \| u\|_{L^{q}} + |v |_\ell.
\end{equation}
For simplicity, we set  $X^\ell= X_0^\ell$. Likewise, the index $\ell$ is omitted when $\ell=0$, that is $X_t=X_t^0$ and $X=X_0^0$.

For some of the proofs below, it will be convenient to use    the natural energy of the problem at time $t$
\begin{equation} \label{energy}
\E_{X_t^\ell}(u,v) =  \mu(t)|u |_{1+\ell}^2 + \frac{2}{q}\|
u\|_{L^{q}}^{q} + |u|_\ell^2+ |v |_\ell^2,
\end{equation}
in place of the $X_t^\ell$-norm. Indeed, from the elementary relations
$$
a^2 + b^q  \leq a^2+b^2 + (a^2+b^2)^{\frac q2}, \qquad a^2 + b^2 \leq a^2+b^q + (a^2+b^q)^{\frac 2q}, 
$$
we see that
\begin{equation} \label{energynosense}
\E_{X_t^\ell}(z) \leq \|z\|_{X^\ell_t}^2 +  \|z\|_{X^\ell_t}^q, \qquad   \|z\|_{X^\ell_t}^2 \leq \E_{X_t^\ell}(z) + \E_{X_t^\ell}(z)^{\frac2q}.
\end{equation}
Hence, the energy $\E_{X_t^\ell}(\cdot)$ is equivalent to the norm $  \|\cdot \|_{X_t^\ell}$, in the following sense:
\begin{itemize}
\item[$\cdot$]  a family $\B=\{\B(t) \subset X_t^\ell\}$ is pullback-bounded
if and only if $$\displaystyle \sup_{s\in (-\infty, t]} \sup_{z
\in \B(s)}\E_{X_s^\ell}(z)< \infty \qquad \forall t \in \R;$$
\item[$\cdot$] a sequence $\{z_n\} \subset X_t^\ell$ converges to $z \in X_t^\ell$ if and only if $\E_{X_t^\ell}(z_n-z)$ converges to zero.
\end{itemize}
In accordance with the above notation we write $\E_{X_t}$ when $\ell=0$.
\section{Main results}
This section contains the main results of the article. Unless otherwise specified, the assumptions of Subsections \ref{Sect3subnl},  \ref{Sect3subtd} and \ref{Sect3subfs} are standing. Additional assumptions will be specified as needed.
\subsection{Well-posedness and dissipativity.}
Our first theorem is a well-posedness result in the base spaces $X_t$. This theorem will also clarify the role of the decay rate $\eps_\omega$ introduced in \eqref{epsom}.
\begin{theorem}    Problem \eqref{SYS} generates a strongly continuous process
$S(t,s): X_s \to X_t$,  
with the following   continuous dependence
property: for every pair of initial conditions
$z^\imath  \in X$ $(\imath=1,2)$ with
$\E_{X_s}(z^\imath) \leq R$ and every $t \geq s$, we have
\begin{equation}
\label{continuous:dep} \E_{X_t}\big[S(t,s)z^1-S(t,s)z^2\big] \leq
 \exp\left(\Q_1(R) \big((t-s)+ \textstyle\int_{s}
^t \frac{1}{\mu(\tau)} \, \d \tau \big)\right)\E_{X_s}\big[z^1-z^2\big] . \end{equation} 
Moreover, there exists $R_\AA =R_\AA(\omega,a_\imath)>0$ such that
the family 
\begin{equation} \label{radius}
\AA=\big\{\AA(t)= \{z \in X_t : \E_{X_t}(z) \leq
R_{\AA} \}\big\}
\end{equation} is an absorber for the process
$S(\cdot,\cdot).$
The dependence of $R_\AA$ and $\Q_1$ on the physical parameters of the problem is specified in the proof. \label{wp:th}
\end{theorem}
The proof of Theorem \ref{wp:th} is given in Section \ref{proof:wp}.
\subsection{Existence of the global attractor}
 We first obtain the existence and uniqueness of the attractor without any additional smoothness.  Smoothness questions will be addressed in a subsequent result, imposing further assumptions on  $\mu$.

In the critical case $q=4$, we require a slight strengthening of the assumptions on the nonlinear term. 
Following \cite{GP}, we ask for the existence of a splitting
 $\varphi=\phi + \psi$, with $\phi,\psi \in \C^2(\R)$, and, for some $2<\gamma <4$, \begin{align} \label{H2.a} & \phi'(y) \geq \tilde{a}_0|y|^2, \quad |\phi''(y)| \leq c(1+|y|); \tag{H2.a} \\ \tag{H2.b} & |\psi'(y)| \leq \tilde{c} (1+|y|^{\gamma-2}) \label{H2.b}
\end{align}
\begin{remark} For polynomial-type nonlinearities $\varphi$ fulfilling (H0)-(H1), the existence of a decomposition of the type (H2.a)-(H2.b) is achieved by choosing $\phi$ to be  the leading term in $\varphi$.
\end{remark}

\begin{theorem}
 \label{osc:attractor}  In addition to the hypotheses (H0)-(H1) and (M1) of Section 3, assume also, when $q=4$, that (H2) holds. Then the family $\A(t)=\omega_\AA(t)$ is the unique (in the sense of Remark~\ref{uniqueness}) global attractor
of the process $S(\cdot, \cdot)$ generated by (P). 
\end{theorem}
\begin{remark}  {Corollary \ref{cor:uniqueness} tells us that $\A(t)\subset\AA(t)$ for every $t \in \R$. More explicitly, there holds the estimate\begin{equation}
\label{bootstrap:start}
\|S(t,s)z\|_{X_t} \leq \Q(R_\AA), \qquad \forall s \in \R,\, z \in \A(s),\, t \geq s,
\end{equation} which   will be of use later.}
\end{remark}
Theorem \ref{osc:attractor} is proven in Section \ref{proof:osc:attractor}.

\subsection{Regularity properties of the global attractor}
In order to derive additional regularity properties of the global attractor $\A$ of  Theorem \ref{osc:attractor} we need additional assumptions on the time-dependent squared speed of propagation $\mu$.
We assume the following two conditions. 
\begin{itemize}
\item[(M2)] there exist constants $C\geq0, \theta \in [0,1)$ such that
\begin{equation}\label{pospart}
\int_{t_1}^{t_2} \frac{(\mu')_+(t)}{\mu(t)} \, \d t \leq C\big(1+(t_2-t_1)^\theta \big), \qquad \forall t_1 \leq t_2,
\end{equation}
where $(\mu')_+$ stands for the positive part;
\item[(M3)] the function  $\nu=\frac 1 \mu$ belongs  to $L^{\infty}(-\infty,t)$, for every $t \in \R$, i.e.
\begin{equation} \label{cicnu}
\cic{\nu}(t):=\|\nu\|_{L^\infty(-\infty,t)} < \infty, \qquad \forall t \in \R.
\end{equation}\end{itemize}

\begin{theorem}\label{thm:reg} We supplement the hypotheses of Theorem \ref{osc:attractor} with \emph{(M2)-(M3)}.
Then the global attractor $\A=\A(t)$ constructed in Theorem \ref{osc:attractor} possesses the additional regularity
$$
\|\A(t)\|_{X^1_t} \leq h_1(t) \qquad \forall t \in \R,
$$
where $h_1$ is a positive increasing continuous function which depends on the physical parameters of the problem (in particular on $\cic{\nu}$) and which can be explicitly computed.
\end{theorem}
The proof of Theorem \ref{thm:reg} is presented in Section \ref{proof:thm:reg}. Here, let us comment and motivate the additional conditions (M2)-(M3).
\begin{remark}\label{rem:M2} Condition (M2) implies additional integrability for the time-derivatives of the solution (see Lemma \ref{rem:addint} below). We describe some relevant  qualitative behaviors falling  inside the scope of assumptions (M1)-(M3).

\vskip1.5mm
\noindent \emph{The case of a decreasing $\mu$.} 
If $\mu$ is a decreasing function on $\R$, conditions (M2), with $C=0$, and (M3) hold true. This ensures that the  expanding universe model of \cite{FGGIRS}, as described in Remark \ref{physcond}, fits into the above assumptions (M1)-(M3), since $\mu$ is positive decreasing by \eqref{graham-hyp}.   
 \vskip1.5mm
\noindent \emph{Finitely many critical points.} Assume, together with (M1) and (M3), that the set $$\mathcal{I}_\mu=\{t \in \R: \mu'(t)>0\}$$(i.e.\ the set on which $\mu$ increases) is the union of finitely many intervals $(t_{i}^{\ell}, t_{i}^{r} )$,  $i=1,\ldots,Z$, $t_{i+1}^r <t_{i}^\ell$,   and possibly $t_Z^\ell =-\infty$. If  $t_Z^\ell =-\infty$, assume further that there exist $\delta>0$ and $\vartheta \in [0,1)$ such that
\begin{equation} \label{M21}
\mu(t_2) \leq \exp(\delta(t_2-t_1)^\vartheta)\mu(t_1)
\end{equation}
holds for each $\forall t_1 \leq t_2 \leq t_Z^r$.
Then \eqref{pospart} holds, for some positive constant $C$ depending on $\mu$, and with $\theta=0$ if $t_Z^\ell>-\infty$, or with   $\theta=\vartheta$ appearing in \eqref{M21}, if $t_Z^\ell=-\infty$.
\vskip1.5mm
\noindent \emph{Oscillating behavior as $s\to -\infty$.} Assume, together with (M1) and (M3), that the set $\mathcal{I}_\mu$  is the union of infinitely many intervals $(t_{i}^{\ell}, t_{i}^{r} )$, $t_{i+1}^r <t_{i}^\ell$, and that there exists $\delta>0,\vartheta \in [0,1)$ such that\begin{center}
\eqref{M21}
holds with   $t_1=t_i^\ell, t_2=  t_i^r$,  $\qquad \forall i=1,2,\ldots$
\end{center} Furthermore, assume that $T_i=t_i^r-t_i^\ell$ satisfy the summability condition
\begin{equation}
\label{summ}
\sum_{i=1}^\infty \frac{1}{T_i} = B< \infty.
\end{equation}
Then \eqref{pospart} holds, with  $C=cB^{\frac{1-\vartheta}{2}}$, and $\theta=\frac{1+\vartheta}{2}$.

We postpone to   Remark \ref{remarkY} the verifications that these assumptions are sufficient for \eqref{pospart}, and hence (M2), to hold. 
\end{remark}

\subsection{Finite-dimensionality of the global attractor}
For a compact subset $K$ of a Banach space $X$, define the \emph{fractal dimension}\footnote{For more details on the fractal dimension (also known as the Minkowski or \emph{box-counting} dimension), we refer the reader to e.g.\ \cite{MAN,SCH}; see also \cite{TEM}.} of $K$ in $X$ as
$$
\dim_{X} K= \limsup_{\eps \to 0^+} \frac{\log \mathcal
N_\eps(K,X)}{\log \textstyle \frac1\eps}
$$
where $ \mathcal N_\eps (K,W) $ 
indicates the minimum number of  balls of $X$ of radius $\eps$ covering
$K$.

The final result is that, under the same assumptions as for Theorems \ref{osc:attractor} and \ref{thm:reg}, the sections $\A(t)$ of the pullback global attractor $\A$ constructed therein have finite fractal dimension, as stated in the next theorem.

\begin{theorem}  \label{thm:fd}Under the assumptions \emph{(H0)-(H2)} on $\varphi$, and \emph{(M1)-(M3)} on $\mu,\omega$, the sections of the pullback global attractor $\A$ of Theorems \ref{osc:attractor} and \ref{thm:reg} have finite fractal dimension:
$$
\dim_{X_t} \A(t)\ \leq h_2(t), \qquad  \forall  t \in \R,
$$
where  the positive increasing function $h_2$ depends only on
the physical parameters of the problem  and can be explicitly
computed.
\end{theorem}
The final Section \ref{proof:thm:fd} contains the proof of Theorem \ref{thm:fd}.



\section{Proof of Theorem \ref{wp:th}} \label{proof:wp}
We begin by deriving a suitable a-priori dissipative estimate for the solution, as stated in the following lemma.
\begin{lemma}
Let $z  \in X$, and $S(t,s)z$ be the solution of \eqref{SYS} with initial time $s \in \R$ and initial data $z$. The following a-priori estimate holds:
 \begin{equation}
\label{two0} \E_{X_t}(S(t,s)z)  \leq K_0\E_{X_s}(z)
\e^{-\eps_\omega(t)(t-s)} + K_1, \qquad \forall t\geq s,
\end{equation}
with $ K_1
=8{c_1}^{-1}(c_0+b_1)$, $c_0$, $b_1$ from \eqref{fritto}  and \eqref{bound:V},  $c_1$ defined below. The  positive constants $c_1,K_0$,
explicitly defined in the proof below,  depend only on the physical
parameters $W$, $a_\imath$ and $q$.
\end{lemma}
\begin{proof}   
 Hereafter, $(u(t),\pt u(t))$ denotes the solution to \eqref{SYS} with initial time $s \in \R$ and initial condition $z=(u_0,v_0) \in X, $ which we assume to be sufficiently regular. 

A  multiplication of (P) by $\pt u$ entails
\begin{equation}
\label{one} \ddt \left[\mu|u |_1^2 + |\pt u |^2  + 2 \V(u)
\right] -\mu'|u |_1^2 + 2\omega|\pt u |^2  =0,\end{equation}
while multiplying (P) by $u$ and then using \eqref{fritto} yields
\begin{equation}
\label{two} \ddt  \left[ \omega| u|^2+ 2\l \pt u,  u  \r \right] +
2\mu| u |_1^2 -\omega'|u|^2 - 2|\pt u |^2 = -2 \l\varphi(u),u \r \leq
-2\V(u) + 2c_0.
\end{equation} 
For $\eps>0$ to be determined later, we add
\eqref{one} to $2\eps$-times \eqref{two}. Setting
\begin{align*}
&\Lambda  =  \mu|u |_1^2 + |\pt u |^2  + 2\V(u) + 2\eps( \omega| u|^2+ 2\l \pt u,  u  \r ), \\
& \Lambda_\star = (2\eps\mu -\mu') | u |_1^2 + (\omega-6\eps)|\pt
u|^2 +(-2\eps \omega' -4\eps^2\omega) |u |^2-4\eps^2 \l\pt u, u \r,
\end{align*}
we obtain
\begin{equation}
\label{three} \ddt \Lambda + 2\eps \Lambda  + \Lambda_\star + \omega |\pt
u|^2   \leq 4\eps c_0.
\end{equation}
Let us now fix $t_0 \geq s$.
By restricting ourselves to (say) $\eps \leq \min\big\{\frac{1}{4},\frac{b_0}{2}\big\}$, we claim   the bound
\begin{equation}
\label{bound} c_1\E_{X_t}[(u(t),\pt u(t))] - 2b_1 \leq \Lambda (t) \leq
c_2\E_{X_t}[(u(t),\pt u(t))]
\end{equation}
with $c_1= \min\{qb_0,1\}/2$ and $c_2$  a positive constant
depending (increasingly) on $W$ and $b_2$. Indeed, the left-hand side bound comes from \eqref{bound:V}, and by applying the Cauchy-Schwarz inequality combined with the restriction on $\eps$;  the right-hand bound simply follows from \eqref{bound:V}.
We now further restrict $\eps$ in order to control $\Lambda_\star$ from below.
We claim that, if we choose
$$
\eps=\eps_\omega(t_0)=\textstyle \frac{1}{16}\min\left\{1, \omega(t_0),\frac{c_1}{1+W}  \right\},
$$
(as in \eqref{epsom}) then
\begin{equation}
\label{boundstar}
\Lambda_\star(t) \geq -\eps \Lambda(t)-2\eps b_1, \qquad \forall t \in [s,t_0]. \end{equation}
 Indeed, a consequence of the lhs bound in \eqref{bound} is that 
$$
|u (t)|^2 \leq \E_{X_t}[(u(t),\pt u(t))]  \leq c_1^{-1}[\Lambda(t) + 2b_1];
$$
therefore, using assumption \eqref{assmu} to control the first term on the rhs, and also recalling that  $\omega$ is decreasing, 
\begin{align*}
\Lambda_\star & \geq
(2\eps\mu -\mu') | u |_1^2 + (\omega-6\eps-4\eps^2)|\pt
u|^2 +(-2\eps \omega' -4\eps^2(1+\omega)) |u |^2
\\ & \quad \,\, (\textrm{with } \mu'(t) \leq 2\alpha(t)\mu(t) \leq 2 \eps_\omega(t_0) \mu(t) = 2\eps \mu(t),\; t \leq t_0)\\
& \geq -4\eps^2 (1+W)|u |^2  \geq -4\eps^2 (1+W)c_1^{-1}\Lambda -
2\eps^2 b_1c_1^{-1} \\ & \geq - \eps \Lambda-2\eps b_1 ,
\end{align*}
as claimed. The above turns \eqref{three}
into
\begin{equation}
\label{three:a}  \ddt \Lambda + \eps \Lambda
\leq 4\eps (2c_0+b_1).
\end{equation}
Multiplying the above inequality by $\e^{\eps t}$ and integrating
between $s$ and $t_0$, we obtain
\begin{equation}
\label{four}  \Lambda(t_0) \leq  \Lambda(s) \e^{-\eps(t_0-s)} +
4(2c_0+b_1);
\end{equation}
an exploitation of \eqref{bound} then leads to
\begin{align}
  c_1\E_{X_t}[(u(t_0),\pt u(t_0))]  &\leq  \Lambda(t_0) + 2b_1 \nonumber \\ &
\leq  \Lambda(s) \e^{-\eps(t_0-s)} + 8(c_0+b_1)  \nonumber \\ &
 \leq   c_2\E_{X_s}( z)  \e^{-\eps(t_0-s)}   + 8(c_0+b_1), \label{four:a}
\end{align}
which is \eqref{two0}, with $K_0=c_2/c_1, K_1
=8c_1^{-1}(c_0+b_1).$ Note that, like $b_1$ and $c_0$, $K_1=0$
when $a_1=0$ in (H1). This completes the proof of \eqref{two0}.
\end{proof}
Having  \eqref{two0} at our disposal, global existence of (weak)
solutions $(u(t),\pt u(t))$ to Problem \eqref{SYS}  is obtained by
means of a standard Galerkin scheme. The solutions we obtain in
this way satisfy, on any interval $(s,t)$, $-\infty<s<t<+ \infty,$
$$
u \in L^\infty\big(s,t; H^1\big) \cap
L^q\big(s,t;L^q(\Omega)\big), \quad \pt u \in L^\infty\big(s,t;
H\big).
$$
Replacing $L^\infty$ on $(s,t)$ with continuity on $[s,t]$
requires some additional work, as explained in \cite[Section
II.4]{TEM}.

 \vskip1.5mm
 Uniqueness of solutions, and
therefore generation of the process $S(t,s)$ will then follow once
the continuous dependence estimate \eqref{continuous:dep} is
established.

\begin{proof}[Proof of \eqref{continuous:dep}] For $\imath=1,2$, let
$z^\imath=(u_0^\imath,v_0^\imath) \in X$ with $\E_{X_s}(z^\imath)
\leq R$. Accordingly, call $(u^\imath(t),\pt
u^\imath(t))$ the solution corresponding to initial datum $z^\imath$, prescribed at time $s \in \R$. Preliminarily, we recall that the
dissipative estimate  \eqref{two0}   can be rewritten as
\begin{equation}
\label{dependence:1} \E_{X_t}[(u^\imath(t),\pt u^\imath(t))]  \leq K_0 R + K_1:=\Q(R), \qquad \forall
t\geq s.
\end{equation}

Then, we observe that the difference $$\bar z(t) = (u^1(t),\pt u^1(t))-
(u^2(t),\pt u^2(t)) = (\bar u(t),\pt \bar u(t))$$ fulfills the Cauchy
problem on $(s,+\infty)$
$$
\begin{cases}
\displaystyle \ptt \bar u+\omega  \pt \bar u + \mu A \bar u + \bar u|\bar
u|^{q-2}   +\bar  u=   \bar u + \varphi(u^2)-\varphi(u^1) +  \bar u|\bar
u|^{q-2},
 \\
\bar z(s) = z^1-z^2.
\end{cases}
$$
Assuming $(\bar u ,\pt \bar u )$ sufficiently smooth, we multiply the above equation by $\pt \bar u$ and obtain the
differential inequality
\begin{equation} \label{dependence:2}
\ddt \E_{X_t}(\bar z) \leq \mu'|\bar u|_1^2+ 2 \l\bar u + \varphi(u^2)-\varphi(u^1)  + \bar u|\bar
u|^{q-2}, \pt \bar u \r.
\end{equation}
The first term in the right-hand side is bounded,   using \eqref{assmu}, by $\alpha\mu |\bar u|_1^2$, observing that $\alpha$ is bounded by 1.
The second term  is easily bounded by $
2|\bar u||\pt \bar u|$. Regarding the third, in view of (H1), we exploit H\"older's inequality and ($2\leq q \leq 4$) usual Sobolev embeddings and obtain $$
2 \l \varphi(u^2)-\varphi(u^1), \pt \bar u \r  \leq  c \Big(1+  |u^{1}|^2_1+ |u^{2}|^2_1 \Big)|\bar u|_1 |\pt \bar u|.
$$
Treating $|\bar u|^{q-2}$ as done above for $\varphi'$ yields the similar control
$$
2 \l \bar u |\bar u|^{q-2}, \pt \bar u \r \leq     c \Big(1+  |u^{1}|^2_1+ |u^{2}|^2_1 \Big)|\bar u|_1 |\pt \bar u|.
$$
Recalling \eqref{dependence:1}, we  bound
\begin{equation} \label{dependence:2ab}  
|u^{\imath}(t)|_1^2 \leq \mu(t)^{-1} \E_{X_t} (u^\imath(t),\pt u^\imath(t)) \leq \mu(t)^{-1} \Q(R),
\end{equation}
and get the estimate
\begin{equation} \label{dependence:2a}  
\ddt \E_{X_t}(\bar z(t)) \leq c\big(1+\mu(t)^{-1}\big) \Q(R)\E_{X_t}(\bar
z(t)).
\end{equation}
We then apply Gronwall's lemma on $(s,t)$ to obtain
\begin{align*}
&\E_{X_t}(\bar z(t)) \leq   \exp\left(\Q_1(R) \big((t-s)+ \textstyle\int_{s}
^t \mu(\tau)^{-1} \, \d \tau \big)\right) \E_{X_s}(z^1-z^2),
\end{align*}
where $\Q_1(R)=c\Q(R)$, as claimed in \eqref{continuous:dep}, so that the proof is complete.
\end{proof}
\begin{proof}[Conclusion of the proof of Theorem \ref{wp:th}] We are only left to  show that the family $\AA$ defined in \eqref{radius} is pullback-absorbing for the process $S(t,s)$, with a suitable choice of $R_\AA$ specified below. Let $\B$ be a pullback-bounded family and, for $t \in \R$, let
$$
R(t)=\sup_{s \in (-\infty,t]}\E_{X_s} [\B(s)],
$$
which is finite for every $t$, due to the equivalence between the energy $\E_{X_t}$ and the $X_t$-norm.
Estimate \eqref{two0} then reads
$$
\E_{X_t}(S(t,s)z) \leq K_0 \E_{X_s}(z) \e^{-\eps_\omega(t)(t-s)} + K_1 \leq
K_0 R(t) \e^{-\eps_\omega(t)(t-s)} +   K_1 \leq 1 +2K_1
$$
for every $z \in \B(s)$, provided that
\begin{equation} \label{enteringtime}
 s \leq t_0=t_0(t): =t-\max\big\{0,(\eps_{\omega}(t))^{-1}\textstyle\log \frac{K_0 R(t)}{1+K_1} \big\}.
\end{equation}
Taking the supremum over $z \in \B(s)$ , we obtain
$$
\E_{X_t}[S(t,s)\B(s)]  \leq 1+2K_1, \qquad \forall s \leq t_0,
$$
which, setting $R_\AA=1+2K_1 $, reads exactly $S(t,s)\B(s) \subset
\AA(t)$ whenever $s \leq t_0(t)$. This ensures that $\AA$ is  pullback-absorbing for the process $S(t,s)$, and concludes the proof of Theorem \ref{wp:th}.
\end{proof}

\begin{remark} The radius $R_{\AA}$ of the absorber $\AA(t)$ does not depend on $t$; however, the entering time of a pullback-bounded family $\B$ into $\AA(t)$ depends explicitly on $t$ (see \eqref{enteringtime}), unless $\eps_{\omega}$ in \eqref{epsom} is uniformly bounded from below.
\end{remark}

\section{Proof of Theorem \ref{osc:attractor}} \label{proof:osc:attractor}
We will entirely devote ourselves to the proof of the (critical) case $q=4$. The proof in the case $q<4$, where no additional assumptions are needed, can be handled in a much simpler way  along the same lines.   
We will work
throughout  with
$$
z=(u_0,v_0) \in \AA(s);
$$
until the end of the section, the generic constants $c>0$
appearing below depend only on $R_\AA$, whose dependence   on the
physical parameters of the problem has been specified earlier.
Hence, the estimate \eqref{two0} now reads
\begin{equation}
\label{palmiro} \E_{X_t}(S(t,s) z)  \leq  K_0 R_\AA + K_1 :=c_2 ,
\qquad \forall s \in \R, t     \geq s.
\end{equation}

We decompose the solution of Problem (P) into 
\begin{equation} \label{decompos}
(u(t),\pt u(t)) = S(t,s)z = P_{z}(t,s)  + N_z(t,s) = (p(t),\pt p(t)) + (n(t),\pt n(t)),
\end{equation}
where
\begin{equation}
\label{SYS-P}
\begin{cases}
\displaystyle
\ptt p +\omega   \pt p  + \mu  A p  + 2  p  +\phi(p)=0, & t \geq s,\\ \\
p(s)= u_0 , \pt p(s) = v_0,
\end{cases}
\end{equation}

\begin{equation}
\label{SYS-Q}
\begin{cases}
\displaystyle
\ptt n +\omega   \pt n + \mu A n + \varphi(u) -\phi(p) = 2  p, & t \geq s,\\ \\
n(s)= 0, \pt n(s) = 0.
\end{cases}
\end{equation}
\begin{lemma} \label{lemma:P}
There exists  $K_2>0$ such that
\begin{equation}
\label{SYS-Pa}
\E_{X_t}(P_z (t,s)) \leq K_2R_\AA \e^{-\tilde{\eps}_\omega(t) (t-s)}  \leq K_2R_\AA  \qquad \forall
t \in \R, s \leq t,
\end{equation}
where $\tilde{\eps}_\omega $ is given by \eqref{epsom}, with $c_1$ replaced by $\tilde c_1>0$, which is specified below.
\end{lemma}
\begin{proof}
We peruse the proof of  \eqref{two0}, Theorem \ref{wp:th},
replacing $\varphi $ with  $\tilde\varphi(y)=2y+\phi(y)$. From (H2.a) we read that (H1)
holds with $a_1=0,$ and the corresponding potential $\tilde{V}(y)$ satisfies \eqref{fritto} with $c_0=0$, and
\eqref{bound:V} with (e.g.) $b_0=\frac{1}{4},b_2=1$, and $b_1=0$. Incidentally, $\tilde c_1=\min\{6b_0,1\}/2 $. Hence
$K_1=0$ in \eqref{two0}, which is exactly the claimed estimate.
Observe that the constant  $K_2$ can be explicitly
computed.\end{proof}

\begin{remark} 
We  observe that $n(t)=u(t)-p(t)$, so that, using \eqref{palmiro} and
Lemma~\ref{lemma:P},
\begin{equation} \label{lemma:Q-1}
\|n(t) \|^4_{L^4} + \mu(t) | n(t) |_1^2+ |\pt n(t)|^2 \leq
c.
\end{equation}

\end{remark}

In the next lemma, we will derive  (formally) certain differential inequalities for some energy functionals involving the solution of \eqref{SYS-Q}. These inequalities will be used both to obtain the compactness of the solution operator $N_z(t,s)$ and to conclude the proof of Theorem \ref{osc:attractor}, and, in the subsequent Section 7, to obtain  a regularity estimate for the global attractor.  
 \begin{lemma}For a given $0\leq \ell \leq 1$, and for given $ v \in H^{1+\ell},w\in H^{\ell}, t \geq s$, we define the functionals
\begin{align*}
 &\Lambda_1^\ell(v,w;t)=\mu(t)|v |_{1+\ell}^2 + | w|_\ell^2 + 2 \l \varphi(u(t))-\phi(p(t))-p(t), A^{\ell} v \r,&
 \\ & \Lambda_{2}^\ell (v,w;t) = \omega(t) |v|^2_{\ell} + 2 \l v , w\r_\ell ,
 \\ & \Lambda_{3}^\ell (v,w;t) = \Lambda_{1}^\ell (v,w;t) + 2\eps_\omega(t) \Lambda_{2}^\ell (v,w;t).\end{align*} 
For a fixed $\sigma \in [0,1]$, let  $\eta=\eta(\sigma)=\min\{\frac14, 2-\frac\gamma 2, 1-\sigma\}$, so that $\eta+\sigma\leq1$. Then, for every $t \geq s$, we have the bounds
\begin{align} 
\label{lemma:Q-ctrl1} &\textstyle \frac{\mu(t)}{2}|n(t) |_{1+\sigma+\eta}^2 + | \pt n(t)|_{\sigma+\eta}^2 - c\left(1+\textstyle\frac{1}{\mu(t)^{4}}\right) \leq \Lambda_1^{\sigma+\eta} (n(t),\pt n(t);t) \\ &\leq  2\mu(t) |n(t) |_{1+\sigma+\eta}^2 + | \pt n(t)|_{\sigma+\eta}^2 + c\left(1+ \textstyle\frac{1}{\mu(t)^{4}}\right), \nonumber
\end{align}
\begin{equation} \label{lemma:Q-8a} \textstyle
-\frac{1}{2} |\pt n(t)|_{\sigma+\eta}^2 -\frac{c}{\mu(t)}\leq \Lambda_2^{\sigma+\eta}(n(t),\pt n(t);t) \leq c |\pt n(t)|_{\sigma+\eta}^2 +  \frac{c}{\mu(t)},
\end{equation}
and the differential inequality
\begin{align}  
&\ddt \Lambda_3^{\sigma+\eta} (n(t),\pt n(t);t)  + 2\eps_\omega   \Lambda_3^{\sigma+\eta} (n(t),\pt n(t);t)    \nonumber  \\ &\leq 
 c\left(1+ \textstyle\frac{1}{\mu(t)^{4}}\right)\big[|\pt u(t)|_\sigma+ |\pt p(t)|+\mu(t)|p(t)|_1^2\big] \Lambda_3^{\sigma+\eta} (n(t),\pt n(t);t) \nonumber \\ & \quad +  c\left(1+ \textstyle\frac{1}{\mu(t)^{6}}\right)\big[1 +|\pt u(t)|_\sigma+ \mu(t)|p(t)|_1^2\big].  \label{thediffin}
 \end{align} \label{superlemma}
 \end{lemma}
\begin{proof}
With  a slight  abuse of notation, we write hereafter $\Lambda_\imath^{\sigma+\eta}(t):=\Lambda_\imath^{\sigma+\eta}(n(t),\pt n(t);t)$, $\imath=1,2,3$. Let us first establish  \eqref{lemma:Q-ctrl1}-\eqref{lemma:Q-8a}. The bound
\eqref{lemma:Q-ctrl1}
 comes from bounding $\Lambda_1^{\sigma+\eta}$ from below, using 
$$
|\l \varphi(u)- \phi(p), A^{\sigma+\eta} n \r | \leq c(1+|u|_1^3+|p|_1^3)|n|_{1+\sigma+\eta} \leq  c\mu^{-1}(1+|u|_1^6+|p|_1^6) + \frac{\mu}{2}|n|_{1+\sigma+\eta}^2,
$$
and  from \eqref{dependence:2ab}, \eqref{palmiro}, \eqref{SYS-Pa} and \eqref{lemma:Q-1}. 
The bound \eqref{lemma:Q-8a} is an easy consequence of \eqref{lemma:Q-1}.

Now, assuming that $(n,\pt n)$ is   regular enough,  we  multiply the
equation \eqref{SYS-Q}
 by $A^{\sigma+\eta}\pt n$, getting  the differential equation
\begin{equation} \label{lemma:Q-2}
\ddt \Lambda_1^{\sigma+\eta} -\mu'|n |_{1+\sigma+\eta}^2 + 2\omega| \pt n|_{\sigma+\eta}^2   =   - 2\l\pt  p -\varphi' (u)\pt u + \phi' (p)\pt p,  A^{\sigma+\eta}  n \r  .
\end{equation} 
The first term in the right-hand side is easily bounded as follows:
\begin{equation} \label{lemma:Q-3}
 -2 \l \pt p, A^{\sigma+\eta}  n \r \leq c|\pt p||A^{\sigma+\eta} n| \leq c|\pt p| | n|_{1+\sigma+\eta}^{\frac{2(\sigma+\eta)}{1+\sigma+\eta}}|n|^{\frac{1-\sigma-\eta}{1+\sigma+\eta}} \leq c|\pt p| |n|_{1+\sigma+\eta}^{2}.  \end{equation} 
For the remaining part, we write 
\begin{equation} \label{lemma:Q-3a}
\varphi' (u)\pt u- \phi'(p)\pt p = \psi'(u)\pt u +n\phi''(\tilde n) \pt u + \phi'(p)\pt n,
\end{equation} where $\tilde n(x,t)$ is chosen between $u(x,t)$ and $p(x,t)$, and therefore satisfies, due to \eqref{palmiro} and  \eqref{SYS-Pa}
\begin{equation} \label{lemma:Q-4}
\|\tilde n(t) \|^4_{L^4} + \mu(t) | \tilde n(t) |_1^2+ |\pt \tilde n(t)|^2 \leq
c.
\end{equation}  
Using the assumption (H2.b) we estimate
\begin{align}
|\l \psi'(u)\pt u, A^{\sigma}  n \r| & 
 \leq c(1+\|u\|^{\gamma-2}_{L^6})    \|  \pt u\|_{L^{\frac{6}{3-2\sigma}}} \|A^{\sigma+\eta} n \|_{L^{\frac{6}{5-\gamma+2\sigma}}}\nonumber \\ & \leq c(1+|u|_1^{\gamma-1}) |  \pt u|_\sigma | n|_{\sigma+2\eta+\frac{\gamma}2 - 1}\nonumber\\& \leq c(1+\mu^{ \frac{1-\gamma}{2} }) |  \pt u|_\sigma |n|_{1+\sigma+\eta} ; \label{lemma:Q-5}
\end{align} 
here we used the embeddings $H^1  \hookrightarrow L^6(\Omega),$ $H^{\sigma} \hookrightarrow L^{\frac{6}{3-2\sigma}}(\Omega)$, $H^{\frac{\gamma}{2}-1-\sigma} \hookrightarrow L^{\frac{6}{5-\gamma+2\sigma}}(\Omega)$, and  the last line follows by using \eqref{palmiro}.
Regarding the term of \eqref{lemma:Q-3a} containing $\tilde n$,  we write, using the appropriate Sobolev embeddings,
\begin{align}
|\l n\phi''(\tilde n) \pt u, A^{\sigma+\eta} n \r| & \leq 
c\|\phi''(\tilde n)\|_{L^6} \|\pt  u\|_{L^{\frac{6}{3-2\sigma}}} \|n\|_{L^{\frac{6}{1-2\eta}}} \|A^{\sigma+\eta }  n \|_{L^{\frac{6}{1+2(\sigma+\eta)}}} 
 \nonumber \\ & \leq c(1+|\tilde n|_1)|\pt u|_\sigma |n|_{1+\eta}|n|_{1+\sigma+\eta}  \leq c(1+\mu^{-1})|\pt u|_\sigma |n|^2_{1+\sigma+\eta};\label{lemma:Q-6}
\end{align}
we used (H2.a),  \eqref{palmiro} and  \eqref{lemma:Q-1} to obtain the rightmost inequality.
For the last term of  \eqref{lemma:Q-3a} we use  (H2.a) and \eqref{lemma:Q-1}:
\begin{align} \nonumber
|\l \phi'(p) \pt n, A^{\sigma+\eta}  n \r|&  \leq 
c\|p^2\|_{L^3} \|\pt n\|_{L^{\frac{6}{3-2(\sigma+\eta)}}} \|A^{\sigma+\eta }   n \|_{L^{\frac{6}{1+2(\sigma+\eta)}}} 
  \\ & \leq c |p|_1^2 |n|_{1+\sigma+\eta}|\pt n|_{\sigma+\eta}. \label{lemma:Q-7}
\end{align}
In view of \eqref{lemma:Q-3} and \eqref{lemma:Q-5}-\eqref{lemma:Q-7}, the rhs of \eqref{lemma:Q-2} is bounded by
\begin{equation} \label{lemma:Q-8}
  c(1+\mu^{-2})(|\pt u|_\sigma+|\pt p|)(1+|n|_{1+\sigma+\eta}^2) + c |p|_1^2 |n|_{1+\sigma+\eta}|\pt n|_{\sigma+\eta}   . \end{equation}
For the functional $\Lambda_{2}^{\sigma+\eta}$, 
multiplying \eqref{SYS-Q} by $A^{\sigma+\eta} n$ yields \begin{align} \nonumber  &
\ddt \Lambda_2^{\sigma+\eta} +2\mu|n |_{1+\sigma+\eta}^2 - \omega'| \pt n|_{\sigma+\eta}^2   +2\l \varphi(u)-\varphi(p) ,A^{\sigma+\eta} n\r  = 2\l p,A^{\sigma+\eta} n\r  \\ & \quad\leq c|p|_{\sigma+\eta-1}|n|_{1+\sigma+\eta} \leq c\mu^{-1} + \mu|n|_{1+\sigma+\eta}^2,  \label{lemma:Q-9}  \end{align}
making use of \eqref{SYS-Pa} in the last inequality.
Combining \eqref{lemma:Q-2} with  \eqref{lemma:Q-9}, and using the bounds  \eqref{lemma:Q-ctrl1} and \eqref{lemma:Q-8a}     yields \begin{align*}  
&\ddt \Lambda_3^{\sigma+\eta}  + 2\eps_\omega   \Lambda_3^{\sigma+\eta} +  (2\eps_\omega\mu-\mu')|n |_{1+\sigma}^2    \\ & \leq 
 c(1+\mu^{-2})(|\pt u|_\sigma+ |\pt p| +\mu|p|_1^2) \Lambda_3^{\sigma+\eta}   +  c(1+\mu^{-6})(1 +|\pt u|_\sigma+ \mu|p|_1^2). 
 \end{align*} 
The last term on the first line is nonnegative, by \eqref{assmu},   so that \eqref{thediffin} follows. This concludes the proof of the lemma.
\end{proof}
\begin{lemma}
There exists  a continuous positive function $h$  such that \label{lemma:Q}
\begin{equation} \label{lemma:Q-0}
  \|N_z(t,s)\|_{X^{\eta}_t}^2 \leq  h(t,s) \qquad \forall t \in \R, s \leq t,
\end{equation}
with $\eta=2-\frac\gamma 2$.
\end{lemma}
\begin{proof}
Having at our disposal  \eqref{lemma:Q-1}, the only terms  in   $\|N_{z}(t,s)\|_{X^{\eta}_t}$ we are left to bound are 
$$
\mu(t)|n(t) |_{1+\eta}^2 + | \pt n(t)|_\eta^2.
$$
To this aim, we use Lemma \ref{superlemma} with $\sigma=0$: again we abuse notation and write $\Lambda_3^\eta(t)$ in place of $\Lambda_3^{\eta} (n(t),\pt n(t);t)$ (and sometimes omit the $t$).   We have, from \eqref{palmiro} and Lemma \ref{lemma:P},
$$
|\pt u(t)| + |\pt p(t)|+\mu(t)|p(t)|_1^2 \leq c,
$$
so that 
\eqref{thediffin} reads, for $t \geq s$,
$$
\ddt \Lambda_3^{\eta}   + 2\eps_\omega   \Lambda_3^{\eta}      \leq 
 c\left(1+ \textstyle\frac{1}{\mu^{4}}\right)  \Lambda_3^{ \eta}   +  c\left(1+ \textstyle\frac{1}{\mu^{6}}\right).
$$
Observe that $\Lambda_3^\eta(n(s),\pt n(s);s)=0$, so that   Gronwall's lemma on the interval $(s,t)$  and the controls \eqref{lemma:Q-ctrl1}-\eqref{lemma:Q-8a}  yield
\begin{align} \label{lemma:Q-final}
\mu(t)|n(t) |_{1+\eta}^2 + | \pt n(t)|_\eta^2  &\leq c\Lambda_3(t) +  c\left(1+ \textstyle\frac{1}{\mu(t)^{4}}\right) \\	&\leq  c\int_s^t  \left(1+ \textstyle\frac{1}{\mu(\tau)^{6}}\right)  \e^{c(t-\tau)+\left(   \int_\tau^t \frac{1}{\mu(y)^4}\d y \right)}   \d \tau +\textstyle  c\left(1+ \textstyle\frac{1}{\mu(t)^{4}}\right). \nonumber
\end{align}
 This concludes the proof of Lemma \ref{lemma:Q}, and the function $h(t,s)$ is given by the last line of  
\eqref{lemma:Q-final}.
\end{proof}

We can now complete the proof of the main theorem.
\begin{proof}[Conclusion of the proof of Theorem \ref{osc:attractor}]
Let $\eta = 2-\frac\gamma2>0$ as in Lemma \ref{lemma:Q}. Observe that $X^\eta_t$ is compactly
embedded in $X_t$ and  
each $X^\eta_t$ is a reflexive Banach space, so that closed balls of
$X^\eta_t$ are closed in $X_t$. These considerations ensure that we are
in  position to apply Corollaries \ref{cor:attractor} and
\ref{cor:regularity}.

 Setting $$
\mathfrak{P}(t,s) = \bigcup_{z \in \AA(s)} P_z(t,s)[z], \qquad \mathfrak{N}(t,s) =
\bigcup_{z \in \AA(s)} N_z(t,s)[0],
$$
we have $S(t,s) \AA(s) \subset \mathfrak{P}(t,s)+ \mathfrak{N}(t,s)$. Lemma \ref{lemma:P} gives
$$
\lim_{s\to -\infty} \|\mathfrak{P}(t,s) \|_{X_t} = \lim_{s \to -\infty}
\sup_{z \in \AA(s)} \|P_z(t,s) \|_{X_t} =0,
$$
while Lemma \ref{lemma:Q} shows that $$ \|
\mathfrak{N}(t,s) \|_{X^\eta_t}^2 \leq h(t,s).
$$ Therefore, $\mathfrak{N}(t,s) $ is compact in $X_t$, for every $t \in \R, s \leq t$. Applying Corollary \ref{cor:attractor}, we obtain  the existence of  the unique pullback-bounded global attractor $\A(t)=\omega_\AA(t)$. \end{proof}


\section{Proof of Theorem \ref{thm:reg}} \label{proof:thm:reg}
Before entering the proof of this theorem, in Lemma \ref{rem:addint} we show how (M2) implies additional integrability of the time-derivatives of the solutions of \eqref{SYS}, \eqref{SYS-P} and \eqref{SYS-Q}, in the following lemma.
\begin{lemma} \label{rem:addint} Let $s \in \R$ be fixed, $z \in \A(s)$,   $t_2 \geq t_1 \geq s$, and write $(u(t),\pt u(t))=S(t,s)z.$ We have that
\begin{equation} \label{dissintegr}
\int_{t_1}^{t_2} \omega(\tau) |\pt u (\tau)|^2 \, \d \tau \leq c(1+(t_2-t_1)^\theta). 
\end{equation}
\end{lemma}
\begin{proof}
Going back to the proof of \eqref{two0} and integrating \eqref{one}  between $t_1$ and $t_2$, we obtain
\begin{equation} \label{dissintegrA}
\int_{t_1}^{t_2} \omega(\tau) |\pt u (\tau)|^2 \, \d \tau \leq c + \int_{t_1}^{t_2} \mu'(\tau)|u(\tau)|_1^2. \end{equation}
 Taking advantage of \eqref{bound}  and (M2) in the last inequality, we have that
\begin{align*}
 \int_{t_1}^{t_2} \mu'(\tau)|u(\tau)|_1^2 \,\d \tau&\leq \int_{t_1}^{t_2} {(\mu')}_+(\tau)  |u(\tau)|_1^2 \, \d \tau
 \leq    \left(\sup_{\tau \in [t_1,t_2]}\mu (\tau)|u(\tau)|_1^2\right)\int_{t_1}^{t_2} \frac{{(\mu')}_+(\tau)}{\mu(\tau)} \, \d\tau\\ & \leq c(1+(t_2-t_1)^\theta).
\end{align*}
This last inequality, in light of \eqref{dissintegrA}, completes the proof of the lemma.
\end{proof}
\begin{remark} \label{remarkY}
As promised, we show how the conditions described in Remark \ref{rem:M2} imply \eqref{pospart}. Under the assumption that $\mu$ is decreasing on finitely many intervals $(t_i^\ell,t_i^r)$, and $t_Z^\ell >-\infty$, we have, for each $t\leq t_Z^r$,
\begin{equation} \label{remarkY:1}
\int_{t}^{+\infty} \frac{{(\mu')}_+(\tau)}{\mu(\tau)} \, \d\tau \leq \sum_{i=1}^Z \big(\log( \mu(t_i^r))-\log({\mu(t_i^\ell)})\big):= C_\mu <\infty,
\end{equation}  
which clearly suffices for (M2). In the case $t_Z^\ell=-\infty$, in view of \eqref{remarkY:1}, it is enough to observe that, when $t_1<t_2\leq t_Z^r,$
$$
\int_{t_1}^{t_2} \frac{{(\mu')}_+(\tau)}{\mu(\tau)} \, \d\tau = \log( \mu(t_2))-\log({\mu(t_1}) \leq c(t_2-t_1)^\theta,
$$
thanks to   condition \eqref{M21}.

We now show how (M2) follows from the \emph{oscillating behavior at $s\to -\infty$} assumptions. Remember that $\mu$ is increasing on an infinite sequence of intervals  $(t_{i}^{\ell}, t_{i}^{r} )$,  $i=1,\ldots,\infty$, $t_{i+1}^r <t_{i}^\ell$, for which \eqref{M21} holds, and the inverses of $T_i=t_i^r-t_i^\ell$ are summable, with $\sum \frac{1}{T_i}=B$.
Let $t_1 \leq t_2$ be fixed and let us consider the $i$'s (if any) for which $(t_i^\ell,t_i^r)$ intersects $(t_1,t_2)$; say $i=j,\ldots,j+k$. We can assume that $t_1<t_{j+k}^\ell<t_j^r < t_2$, the other cases being treated likewise.  By explicit computation and a subsequent use of \eqref{M21},
$$
\int_{t_1}^{t_2} { \frac{(\mu')_+(\tau)}{\mu(\tau)}}\, \d \tau = \sum_{i=j}^{j+k} \log \Big( \textstyle\frac{\mu(t_i^r)}{\mu(t_i^\ell)}\Big)  \leq c\displaystyle\sum_{i=j}^{j+k} (t_i^r-t_i^\ell)^\vartheta.
$$
Therefore, it suffices to show that, for $ \theta=\frac{1+\vartheta}{2}<1$, we have  
$$
\sum_{i=j}^{j+k} (t_i^r-t_i^\ell)^\vartheta\leq   cB^{\frac{1-\vartheta}{2}}(t_2-t_1)^{\theta}.
$$
Using H\"older's inequality with $p=\frac{1}{\theta}$, $p'= \frac{2}{1-\vartheta}$, we have that 
\begin{align*}
\sum_{i=j}^{j+k} (t_i^r-t_i^\ell)^\vartheta & = \sum_{i=j}^{j+k} (t_i^r-t_i^\ell)^{\theta} (T_i)^{\frac{\vartheta-1}{2}}
\\&  \leq 
\bigg(\sum_{i=j}^{j+k} (t_i^r-t_i^\ell) \bigg)^{\theta} \bigg(\sum_{i=1}^{\infty} (T_i)^{\frac{\vartheta-1}{2} \frac{2}{1-\vartheta} } \bigg)^{\frac{1-\vartheta}{2}} \leq cB^{\frac{1-\vartheta}{2}}(t_2-t_1)^{\theta},
\end{align*}
which is the estimate we were looking for. \end{remark}

The main tool of the proof of Theorem \ref{thm:reg} is the bootstrap scheme devised in the proposition below. The idea of relying on  the  time-integrability of the time-derivatives of the solution to obtain regularization effects is fairly common in the literature,  see for example \cite{DP-Hyp,DPZ,GP}. The main novelty of our construction lies in showing that the additional time-integrability is preserved for  a very large class of time-dependent coefficients. As before  in Lemma \ref{superlemma}, for a given $\sigma \in [0,1]$,  we let $$\eta(\sigma)=\textstyle\min\{\frac14, 2-\frac\gamma 2, 1-\sigma\}.$$   From now on, we will always consider initial data $z \in \A(s) $ and write   $S(t,s)z=(u(t),\pt u(t))$. The generic constants $c$ appearing in the proof below will (possibly) depend on $$\sup_{t\in \R}\|\A(t)\|_{X_t} \leq R_\AA,$$ and on the physical parameters of the problem.

\begin{proposition} \label{bootstrap}
Let $\sigma \in [ 0,1]$ be given. Suppose that  there exists a positive increasing  function $$\I^{\mathsf{in}}:\R \to \R,\qquad\I^{\mathsf{in}}(t) \geq c(1+\cic{\nu}(t)^8),   $$ where $\cic{\nu}$ is defined in \eqref{cicnu}, and such that for every $s \in \R$, $z \in \A(s)$, 
\begin{equation}
\label{bootstrap:1}
\|(u(t),\pt u(t))\|_{X^\sigma_t} \leq \I^{\mathsf{in}}(t) \qquad \forall t\geq s.
\end{equation}
Then the solution $P_z(t,s)=(p(t),\pt p(t))$ of  \eqref{SYS-P} satisfies 
\begin{equation}
\label{bootstrap:3}
\int_{t_1}^{t_2} |\pt p(\tau)|_\sigma^2\, \d \tau \leq \J^{\mathsf{out}}(t)\big(1+(t_2-t_1)^{\theta}\big)\qquad  \forall s \leq t_1 \leq t_2 \leq t,
\end{equation}
for some  positive increasing  function$$\J^{\mathsf{out}}:\R \to \R,\qquad\J^{\mathsf{out}}(t) \geq c(1+\cic{\nu}(t)^8).$$ 
In addition to \eqref{bootstrap:1}, suppose that    there exists a positive increasing  function $$\J^{\mathsf{in}}:\R \to \R,\qquad\J^{\mathsf{in}}(t) \geq c(1+\cic{\nu}(t)^8),   $$ such that
\begin{equation}
\label{bootstrap:2}
\int_{t_1}^{t_2} |\pt u(\tau)|_\sigma^2\, \d \tau \leq \J^{\mathsf{in}}(t) \big(1+(t_2-t_1)^{\beta_{\mathsf{in}}}\big)\qquad  \forall s \leq t_1 \leq t_2 \leq t,
\end{equation}
for some   $\beta_{\mathsf{in}}\in [0,1)$.
Then  there exist  positive  increasing functions 
$$\I^{\mathsf{out}}, \K^{\mathsf{out}}:\R \to \R,\qquad\I^{\mathsf{out}}(t),\K^{\mathsf{out}}(t)\geq c(1+\cic{\nu}(t)^8),   $$  such that the solution $N_z(t,s)=(n(t),\pt n(t))$ of \eqref{SYS-Q} satisfies
\begin{equation}
\label{bootstrap:4}
\|(n(t),\pt n(t))\|_{X^{\sigma+\eta(\sigma)}_t} \leq \I^{\mathsf{out}}(t) \qquad \forall t\geq s,
\end{equation}
and
\begin{equation}
\label{bootstrap:5}
\int_{t_1}^{t_2} |\pt n(\tau)|_{\sigma+\eta(\sigma)}^2\, \d \tau \leq \K^{\mathsf{out}} (t)\big(1+(t_2-t_1)^{\beta_{\mathsf{out}}}\big)\qquad  \forall s \leq t_1 \leq t_2 \leq t,
\end{equation}
with $\beta_{\mathsf{out}}=\max\left\{\frac{1+\beta_{\mathsf{in}}}{2},\theta \right\} \in [0,1)$.
\end{proposition}
We postpone the proof of Proposition \ref{bootstrap} to the end of the  section, and now show how Theorem \ref{thm:reg} follows from this proposition.
\begin{proof}[Proof that Proposition \ref{bootstrap} implies Theorem \ref{thm:reg}]
 We will construct a  finite sequence of $\sigma_i$'s: $$0=\sigma_0\leq\sigma_1 \leq \cdots\leq \sigma_\kappa=1, \qquad \sigma_{i+1} =\sigma_i + \eta(\sigma_i), \; i=0,\ldots,\kappa-1,$$ and increasing functions
$$
 \I_{ i}, \J_{ i}:\R \to \R,\qquad\I_{i}(t), \J_{ i} (t)\geq c(1+\cic{\nu}(t)^8), 
$$ such that, for every $s \in \R$, $z \in \A(s)$,
\begin{equation}
\label{bootproof:1} \tag*{(A)$_i$}
\|(u(t),\pt u(t))\|_{X^{\sigma_i}_t} \leq \I_{ i} (t) \qquad \forall t\geq s,
\end{equation}
and
\begin{equation}
\label{bootproof:2} \tag*{(B)$_i$}
\int_{t_1}^{t_2} |\pt u(\tau)|_{\sigma}^2\, \d \tau \leq \J_{i}(t)\big(1+(t_2-t_1)^{\beta_{i}}\big)\qquad  \forall s \leq t_1 \leq t_2 \leq t.
\end{equation}
Note that the number $\kappa$ of steps to go from $\sigma_0=0$ to $\sigma_\kappa=1$ depends only on $\gamma$ in (H2.b) and is always finite. Also, the statement of Theorem \ref{thm:reg} follows from (A)$_\kappa$ and from the invariance of the attractor; the function $h_1$ can be taken equal to $\I_\kappa$ appearing in (A)$_\kappa$.

We now explain how to perform the construction inductively. The base case, that is,   \eqref{bootstrap:1} and \eqref{bootstrap:2} holding true  for $\sigma=0$, follows from \eqref{bootstrap:start} and from \eqref{dissintegr}, with $$\I_0(t)=\J_0(t)=c \max \big\{1, \frac{1}{\omega(t)}, \cic{\nu}(t)^8 \big\}, \qquad\beta_0=\theta.$$

For the induction step, we start by assuming that   \ref{bootproof:1} and \ref{bootproof:2} hold true.    Let then $t$ be fixed and $z \in \A(t)$. Let $z_k \in \A(t-k)$, $k \in \mathbb{N}$,  be chosen so that $S(t,t-k)z_k=z$. The inductive hypotheses show that \eqref{bootstrap:1}-\eqref{bootstrap:2} hold with $\I^{\mathsf{in}}=\I_{i}$, $\J^{\mathsf{in}}=\J_{i}$ so that we can apply Proposition \ref{bootstrap} to $z_k$. We read from  \eqref{bootstrap:4} that   $$\|n_k:=N_{z_k}(t,t-k) \|_{X^{\sigma_{i+1}}} \leq  \I^{\mathsf{out}}(t);$$ due to the compact embedding $X^{\sigma_{i+1}}_t \Subset X^{\sigma_i}_t,$ and to the fact that closed balls of $X^{\sigma_{i+1}}_t $ are closed in $X^{\sigma_i}_t$, $n_k$ has an $X^{\sigma_i}_t$-limit point  $\bar n$ with $\|\bar n\|_{X^{\sigma_{i+1}}_t} \leq \I^{\mathsf{out}}(t)$. However, we have $z=n_k + P_{z_k}(t,t-k)$, and we know from Lemma \ref{lemma:P} that $P_{z_k}(t,t-k) \to 0$ in $X^{\sigma_i}_t$  as $k \to \infty$. We conclude that $z=\bar n$, and therefore $\|z\|_{X^{\sigma_{i+1}}_t} \leq \I^{\mathsf{out}}(t)$. At this point, (A)$_{i+1}$, with $\I_{i+1}(t):=\I^{\mathsf{out}}(t)$, follows  from the invariance of the attractor. In consequence of  (A)$_{i+1}$, we can apply Proposition \ref{bootstrap} with $\sigma_{i+1}$ in place of $\sigma$, $\I_{i+1}(t)$ in place of $\I^{\mathsf{in}}$, and  we also obtain \eqref{bootstrap:3} with $\sigma_{i+1}$ and $\beta_{i+1}$, i.e.
$$
\int_{t_1}^{t_2} |\pt p(\tau)|^2_{\sigma_{i+1}} \, \d \tau \leq \J^{\mathsf{out}_2}(t)   (1+(t_2-t_1)^{\theta}), \qquad  \forall s \leq t_1 \leq t_2 \leq t;
$$
we wrote $\J^{\mathsf{out}_2}$ instead of $\J^{\mathsf{out}}$ to mark   that this output comes from a second application of the proposition. 
We also had from the previous application of the proposition that
$$
\int_{t_1}^{t_2} |\pt n(\tau)|_{\sigma_{i+1}}^2\, \d \tau \leq \K^{\mathsf{out}}(t)\big(1+(t_2-t_1)^{\beta_{\mathsf{out}}}\big)\qquad  \forall s \leq t_1 \leq t_2 \leq t;
$$
(B)$_{i+1}$ now follows by combining the last two bounds, and setting $\J_{i+1}=\K^{\mathsf{out}}+\J^{\mathsf{out}_2}$.  The induction step is complete, and so is the proof that Proposition \ref{bootstrap} implies Theorem \ref{thm:reg}.\end{proof}

\subsection*{Proof of Proposition \ref{bootstrap}}
We  need the following Gronwall-type lemma, adapted {from \cite{GP2}}, which we refer for the proof.
\begin{lemma} Let $\Phi$ be an absolutely continuous positive function on $[s,t_0]$, satisfying the differential inequality
$$
\ddt \Phi(t) + 2\eps \Phi(t) \leq g(t)\Phi(t) + f(t), \qquad \mathrm{a.e.}\; t \in[ s , t_0], 
$$ 
for some $\eps>0$ and where $f,g$ are positive functions on $[s,t_0]$ satisfying
\begin{align} \label{lemmetto:1}
\int_{t}^{\min\{t_0,t+1\}} f(\tau) \, \d \tau \leq F,  &\qquad \forall t \in [s,t_0],\\   \int_{t_1}^{t_2} g(\tau)\,\d \tau\leq G(1+(t_2-t_1)^\beta), & \qquad \forall s \leq t_1 \leq t_2 \leq t_0, \nonumber
\end{align}
for some positive constants $F,G$ and some $\beta\in [0,1)$. Then,
\begin{equation} \label{lemmetto:2}
\Phi(t_0) \leq \Gamma \Phi(s) \e^{-\eps(t_0-s)} + \Theta,
\end{equation}
where $\Gamma=\Gamma(G,\beta,\eps) \geq 1$ is computed explicitly in the proof and $\Theta=\Gamma F \frac{  \e^{\eps}}{1-\e^{-\eps}}$. \label{techlemma}
\end{lemma}\noindent
\begin{proof}[Proof that \eqref{bootstrap:1}-\eqref{bootstrap:2} imply \eqref{bootstrap:4}-\eqref{bootstrap:5}]
The initial time $s \in \R$ and the final time  $t_0 \geq s$ are fixed. As usual, we refer  to Lemma \ref{superlemma}, and write $\Lambda_3^\eta(t)$ in place of $\Lambda_3^{\eta} (n(t),\pt n(t);t)$.  
Setting for brevity 
$$
g_1 =c\left(1+ \nu^{4}\right) |\pt u |_\sigma, \; g_2  = c\left(1+ \nu^{4}\right)\left[|\pt p| + \mu|p|_1^2\right], \; f_1 = c\left(1+ \nu^{6}\right)\big[1 +|\pt u|_\sigma+ \mu|p|_1^2\big],
$$
we rewrite \eqref{thediffin} as
\begin{equation} \label{thediffin-here}
\ddt \Lambda_3^{\sigma+\eta}(t)   + 2\eps_\omega(t_0)   \Lambda_3^{\sigma+\eta}(t)     +\omega(t) |\pt n(t)|^2_\sigma  \leq 
 (g_1(t)+g_2(t))\Lambda_3^{\sigma+\eta} (t)+ f_1(t). \end{equation}
We use our assumption \eqref{bootstrap:2} and condition (M3)  to estimate
\begin{align*}
\int_{t_1}^{t_2} g_1(\tau)\,\d \tau 
& \leq c\left( \int_{t_1}^{t_2} \left(  1+ \nu^{8}(\tau)\right) \,\d \tau \right)^{\frac12} \left( \int_{t_1}^{t_2} |\pt u(\tau)|^2_\sigma \, \d \tau  \right)^{\frac12} \\
& \leq  c(1+\cic{\nu}(t_0)^{4} \J^{\mathsf{in}}(t_0)^{\frac12})\left( 1 +(t_2-t_1)^{\zeta} \right),  
\end{align*}
where we set $ \zeta = \frac{\beta_{\mathsf{in}}+1}{2} <1.$ Thanks to the exponential decay resulting from Lemma \ref{lemma:P}, we have  $$
\int_{s}^{t_0} [|\pt p(\tau)| +\mu(\tau)|p(\tau)|^2_1 ]\,\d \tau
\leq c,$$  so that
\begin{align*}
\int_{t_1}^{t_2} g_2(\tau)\,\d \tau 
& \leq c\left( \int_{t_1}^{t_2} \left(  1+ \nu^{8}(\tau)\right) \,\d \tau \right)^{\frac12} \left( \int_{t_1}^{t_2}\big[ |\pt p(\tau)| + \mu(\tau)|p(\tau)|^2_1\big] \, \d \tau  \right)^{\frac12} \\
& \leq  c(1+\cic{\nu}(t_0)^{4})\left( 1 +(t_2-t_1)^{\frac12} \right). 
\end{align*}
Finally, we use \eqref{bootstrap:1} to control $|\pt u(t)|_\sigma$ pointwise in $t$, obtaining
$$
\int_t^{\min\{t_0,t+1\}} f_1(\tau) \, \d \tau \leq c(1+  \I^{\mathsf{in}}(t_0)^{\frac12})\int_{t}^{\min\{t_0,t+1\}} \left(  1+ \nu^{6}(\tau)\right)\,\d \tau  \leq c(1+  \I^{\mathsf{in}}(t_0)^{\frac12}\cic{\nu}(t_0)^{6}).
$$
We apply Lemma \ref{techlemma}, with $g_1+g_2$ in place of $g$, $f_1$ in place of $f$. Since $\Lambda_3^{\sigma+\eta} (s)=0$, we can write
\begin{align} \label{projboot1}
\mu(t_0)|n(t_0) |_{1+\sigma+\eta}^2 + | \pt n(t_0)|_{\sigma+\eta}^2 &  \leq c\Lambda^{\sigma+\eta} _3(t_0) +  c\big(1+ \nu(t_0)^{4}\big)\\ \nonumber & \leq \Theta (t_0) + c\big(1+ \nu(t_0)^{4}\big),
\end{align} 
where $\Theta(t_0)$ is  the  constant $\Theta$ given by \eqref{lemmetto:2}, and is seen to depend only on $\I^{\mathsf{in}}(t_0),$ $\J^{\mathsf{in}}(t_0),\sigma(t_0),\cic{\nu}(t_0),\eps_\omega(t_0)$.
Observe that all of the above functions are increasing functions of $t_0$, so that  $\Theta$ can be chosen to be increasing in $t_0$ as well.

We turn to \eqref{bootstrap:5}. Let us recall \eqref{lemma:Q-2} from Lemma \ref{superlemma}, which, written for $t \in [s,t_0]$,  reads, in view of \eqref{projboot1},
\begin{align} \label{lemmetto:3}
&\ddt \Lambda_1^{\sigma+\eta} + 2\omega| \pt n|_{\sigma+\eta}^2   \\ \nonumber &\leq \mu'|n |_{1+\sigma+\eta}^2 + c(1+\nu^{3})(|\pt u|_\sigma+|\pt p|+\mu|p|_1^2)(1+\mu|n|_{1+\sigma+\eta}^2 + |\pt n|_{\sigma+\eta}) \\& \leq   c\Theta(t_0)\left(1+ \nu^7\right) \left( \textstyle \frac{\mu'}{\mu} + |\pt u|_\sigma+|\pt p|+\mu|p|_1^2\right). \nonumber
\end{align}
Now, for $s \leq t_1 \leq t_2 \leq t_0$, integrate \eqref{lemmetto:3} on the interval $[t_1,t_2]$. We find, proceeding as before, 
\begin{align*}
&\int_{t_1}^{t_2}   \left(1+ \nu(\tau)^7\right) \left(|\pt u (\tau)|_\sigma+|\pt p(\tau)|+\mu(\tau)|p(\tau)|_1^2\right)\, \d \tau \\ &\leq c(1+\cic{\nu}(t_0)^{7}  \I^{\mathsf{in}}(t_0)^{\frac12})\left( 1 +(t_2-t_1)^{\zeta} \right).
\end{align*}
The remaining term  in the right-hand side of \eqref{lemmetto:3} is bounded  as follows, by using \eqref{pospart}:
$$  \int_{t_1}^{t_2}   \frac{\mu'(\tau)}{\mu(\tau)} (1+\nu(\tau)^7) \; \d \tau  \leq c(1+\cic{\nu}(t_0)^7)(1+(t_2-t_1)^\theta);$$ 
note that this term can be simply neglected when $\mu$ is a decreasing function on $\R$.  
Summarizing, and  using the bound \eqref{lemma:Q-ctrl1} together with \eqref{projboot1},  we finally obtain
\begin{align*}
\omega(t_0)\int_{t_1}^{t_2}   |\pt n(\tau)|^2_{\sigma+\eta} \, \d \tau &\leq \Lambda_1^{\sigma+\eta}(t_0) +  c(1+\cic{\nu}(t_0) + \I^{\mathsf{in}}(t_0) )\left( 1 +(t_2-t_1)^{\max\{\zeta,\theta\}} \right) \\ &  \leq c\Theta (t_0) + c\big(1+ \cic{\nu}(t_0)^{7}+ \I^{\mathsf{in}}(t_0) )\left( 1 +(t_2-t_1)^{\max\{\zeta,\theta\}} \right).   \end{align*}
We compare this last inequality with \eqref{projboot1}, and setting
$$
 \I^{\mathsf{out}}(t_0)=\K^{\mathsf{out}}(t_0) = c \max \big\{1,\textstyle\frac{1}{\omega(t_0)}\big\}\left[\Theta(t_0) + c\big(  \I^{\mathsf{in}}(t_0) + \J^{\mathsf{in}}(t_0)\big)\right], 
$$ and $\beta_{\mathsf{out}} =  \max\{\zeta,\theta\}$,
we obtain \eqref{bootstrap:4}-\eqref{bootstrap:5}.
\end{proof}

\begin{proof}[Proof that \eqref{bootstrap:1} implies \eqref{bootstrap:3}]Fix $s \in \R$ and the final time $t_0 \geq s$. Similarly to Lemma \ref{superlemma},  we define the functionals
\begin{align*}
 &\Psi_1^\sigma(t)=\mu(t)|p(t)|_{1+\sigma}^2 + | \pt p(t)|_\sigma^2 + |p(t)|_\sigma^2  +2\l \phi(p(t)) , A^{\sigma} p(t) \r,&
 \\ & \Psi_{2}^\sigma (t) = \Lambda_2^{\sigma}(p(t),\pt p (t); t) ,
 \end{align*} 
Since $|\l \phi(p) , A^{\sigma} p \r| \leq c|p|_1^3|p|_{1+\sigma}$, we have the usual bounds:
$$
\textstyle \frac{\mu(t)}{2}|p(t) |_{1+\sigma}^2 + | \pt p(t)|_{\sigma}^2 - c\left(1+\nu(t)^{4}\right) \leq \Psi_1^{\sigma} (t)  \leq  2\mu(t) |p(t) |_{1+\sigma}^2 + | \pt p(t)|_{\sigma}^2 + c\left(1+ \nu(t)^{4}\right).
$$
Now, assuming that  $(p,\pt p)$ is   sufficiently regular, we  multiply  the
equation \eqref{SYS-Q}
 by $A^{\sigma}\pt p$. This yields
\begin{align} \label{lemma:intP-2}
\ddt \Psi_1^{\sigma} -\mu'|p |_{1+\sigma}^2 + 2\omega| \pt p|_{\sigma}^2  & =  2\l\phi' (p)\pt p,  A^{\sigma}  p \r  \\ \nonumber  & \leq c\|p^2\|_{L^3}\|\pt p\|_{L^{\frac{6}{3-2\sigma}}} \|A^\sigma p\|_{L^{\frac{6}{1+2\sigma}}} \\ 
\nonumber  & \leq c|p|_1^2 |\pt p|_{\sigma}|p|_{1+\sigma} \leq c\mu|p|_1^2(1+\nu^6)\Psi_1^\sigma.
\end{align} A multiplication by $A^\sigma p$ implies 
\begin{equation} \label{lemma:intP-3}
 \ddt  \Psi_2^\sigma +
2\mu| p |_{1+\sigma}^2 -\omega'|p|_\sigma^2 - 2|\pt p|_\sigma^2 = -2 \l \phi(p),p \r \leq c|p|_1^3|p| \leq c \nu^{\frac32},
\end{equation}
where we used Lemma \ref{lemma:P} in the last inequality. The last two inequalities and the inequality
$$
\textstyle-\frac{1}{2} |\pt p(t)|_{\sigma}^2 -\nu(t)\leq \Psi_2^{\sigma}(t) \leq c |\pt p(t)|_{\sigma}^2 +  \nu(t),
$$
give for $\Psi_3^\sigma(t)=\Psi_1^\sigma(t)+2\eps_\omega(t) \Psi_2^\sigma(t)$ the inequality
\begin{equation} \label{lemma:intP-4}
\ddt \Psi_3^{\sigma} +2\eps_\omega\Psi_3^\sigma  \leq c\mu|p|_1^2(1+\nu^6)\Psi_3^\sigma + c \nu^{\frac32}.
\end{equation}
The exponential decay of $\mu(t)| p(t)|^2_1$ from Lemma \ref{lemma:P} guarantees that
$$
\int_{s}^{t_0} c\mu(t)|p(t)|_1^2(1+\nu(t)^6)\, \d t \leq c(1+\cic{\nu}(t)^6), 
$$
so that an application of the   Gronwall lemma to \eqref{lemma:intP-4} yields the control
$$
\Psi_1^{\sigma}(t) \leq |\Psi_3^{\sigma}(t)| +c|\Psi_2^{\sigma}(t)| \leq c |\Psi_3^{\sigma}(s)| + c\cic{\nu}(t_0)^6\leq c \I^{\mathsf{in}}(t_0), \qquad \forall t \in [s,t_0].
$$
Here, we used that $\I^{\mathsf{in}}$ dominates $c(1+\cic{\nu}^8)$.
Going back to \eqref{lemma:intP-2}, and integrating between $t_1$ and $t_2$ we finally obtain
$$
\int_{t_1}^{t_2} | \pt p(\tau)|_{\sigma}^2\,\d \tau \leq  \J^{\mathsf{out}}(t_0)(1+ (t_2-t_1)^\theta),
$$
provided we set $\J^{\mathsf{out}}(t_0) =c \omega^{-1}(t_0)\I^{\mathsf{in}}(t_0). $
\end{proof}

\section{Proof of Theorem \ref{thm:fd}} \label{proof:thm:fd} 

Theorem \ref{thm:fd} is a consequence of the following proposition, which in turn is closely related to  Lemma 6.1 of \cite {DDT}, see Remark \ref{farlock}.
\begin{proposition} \label{frac:dec}
Suppose that there  
exists a decomposition
\begin{equation} \label{DZdeco}  
 S(t,s)z = D_z(t,s) + K_z(t,s) , \qquad z \in \A(s), \, s\leq t,
\end{equation} such that the following squeezing property holds.\vskip1.5mm \noindent
\textsc{(SP)}\, There exists a positive function $\F$ and a decay rate $\hat\eps_\omega$, depending only on  $h_1$ from Theorem \ref{thm:reg} and on the physical parameters of the problem, such that for each $ t_0 \in \R,\, t_\star >0$, and for every  $s \leq t_0-t_{\star}$, $z^1,z^2 \in \mathcal A(s)$,
\begin{equation} \label{smooth:decay} \tag{SP1}
\|D_{z^1}(s+t_\star,s ) -D_{z^2}(s+t_\star,s )   \|_{X_{s+t_\star}}^2 \leq C\|  z^1-z^2\|_{{X_s}}^2\e^{-\hat\eps_\omega(t_0) t_\star},
\end{equation}
and
\begin{equation} \label{smoooth} \tag{SP2}
\|K_{z^1}(s+t_\star,s) - K_{z^2}(s+t_\star,s) \|^2_{ 
X^1_{s+t_\star}} \leq  \F(t_0,t_\star)  \|   z^1-   z^2 \|_{ X_{s}}^2.  
\end{equation}
   Then Theorem \ref{thm:fd} holds, with $h_2$ depending only on the function $\F$.
\end{proposition}
\begin{remark} \label{farlock} By choosing $t_\star=t_\star(t_0)$ such that 
$
C\e^{-\hat\eps_\omega(t_0) t_\star} \leq \rho < \textstyle \frac14$, 
the decomposition \eqref{DZdeco} satisfies the hypotheses (6.1)-(6.2) of Lemma 6.1 of \cite{DDT}, and therefore, we obtain the conclusion from this lemma. 
\end{remark}
Thus, the remainder of the section is devoted to the verification of the squeezing property (SP). Throughout, it is understood that the   data $z$ at time $s \in \R $  belongs to $\A(s)$, and  we will use the notation
$$
S(t,s) z=\big(u_z(t;s),\pt u_z(t;s)\big), \qquad t \geq s, \,z \in \A(s).
$$
Moreover, the notation $\Q(\cdot)$ will be used to denote a generic positive increasing function of time $t$, depending only on the function $h_1$ from Theorem \ref{thm:reg} and on the physical parameters of the problem.
As a consequence of Theorem \ref{thm:reg}, we have the uniform (in $s$) bound
\begin{equation} \label{fd:bound}  
\sup_{s\leq t \leq t_0}|u_z(t;s)|_2^2 \leq \cic{\nu}(t_0) h_1(t_0).
\end{equation}
The decomposition  \eqref{DZdeco} from Proposition \ref{frac:dec} is achieved as follows:
$$  S(t,s)z= D_z(t,s) + K_z(t,s)  =  (
d_z(t;s),  \pt d_z(t;s) ) + (k_z(t;s), \pt k_z(t;s) ) $$ where
\begin{equation}
\label{SYS-TILDE1}
\begin{cases}
\displaystyle
\ptt d_z+\omega  \pt d_z   + \mu A   d_z+d_z= 0,  \\
D_z(s,s)=z,
\end{cases}
\end{equation}
and
\begin{equation}
\label{SYS-TILDE2}
\begin{cases}
\displaystyle
\ptt k_z+\omega  \pt k_z   + \mu A   k_z=  -d_z - \varphi(u_z),   \\
K_z(s,s)=0.
\end{cases}
\end{equation}
\subsection{Verification of \eqref{smooth:decay}} We peruse once again  the proof of  \eqref{two0}, Theorem \ref{wp:th},
this time replacing $\varphi $ by  $\hat\varphi(y)=y$. It follows that
\begin{equation}
\label{decay1}
\|D_z(t,s) \|_{X_t}^2 \leq   C\|  z\|_{{X_s}}^2\e^{-\hat\eps_\omega(t) (t-s)}, 
\end{equation}
where $\hat\eps_{\omega}$ is given by \eqref{epsom} with $c_1=1$. Since $z \mapsto D_z(t,s)$ is linear and $\hat\eps_\omega$ is decreasing, it follows that 
$$
\|D_{z^1}(t,s ) -D_{z^2}(t,s )   \|_{X_{t}}^2 \leq  C\|  z^1-z^2\|_{{X_s}}^2\e^{-\hat\eps_\omega(t_0) (t-s)}, \qquad \forall s \leq t \leq t_0.
$$
Thus \eqref{smooth:decay} follows, with $t_\star=t-s$.
\subsection{Verification of \eqref{smoooth}} We begin by noting that the   
difference $(k(t;s), \pt k (t;s)) =  K_{z^1}(t,s)  -
K_{z^2}(t,s)$ solves the Cauchy problem
\begin{equation}
\label{SYS-TILDE}
\begin{cases}
\displaystyle
\ptt   k +\omega  \pt k + \mu A k=f  + g,   \\
k(s;s)= 0 , \pt k(s;s) =  0,
\end{cases}
\end{equation}
where $\bar z=z^1-z^2$ and   $$
f(t;s)= -d_{\bar z}(t;s)  , \qquad g(t,s)=-\big(\varphi\big(u_{z^1}(t;s)\big)-\varphi\big(u_{z^2}(t;s)\big)\big).$$
We recall here, as a consequence of the continuous dependence estimate \eqref{continuous:dep}, that \begin{equation} \label{INTER-TILDE}
\| u_{z^1}(t;s)-u_{z^2}(t;s)\|^2_{X_t} \leq \Q(t_0) \exp\big((1+\cic{\nu}(t_0)) (t-s) \big)  \|  \bar z \|_{ X_{s}}^2, \qquad s \leq t \leq t_0.
\end{equation}
With this in hand, we begin the core of the proof of \eqref{smoooth}. First, we derive  appropriate estimates for the  terms $f,g$ in \eqref{SYS-TILDE}. We deduce from \eqref{decay1} that
\begin{equation}
\label{decay2}
|f(t;s) |^2_1=  |d_{\bar z}(t;s) |_1^2 \leq \cic{\nu}(t_0)\|D_z(t,s) \|_{X_t}^2   \leq \Q(t_0)\|  \bar z\|_{{X_s}}^2,\qquad   s \leq t \leq t_0.
\end{equation}
Regarding $g$, we write, omitting the time dependencies
$$
\nabla_x g(t;s) = \varphi'(u_{z^1}) \nabla_x (u_{z^1}-u_{z^2}) + \varphi''(\bar u) (u_{z^1}-u_{z^2}) \nabla_x u_{z^1},$$  for some $\bar u(x,t;s) \leq \max\{u_{z^1}(x,t;s),u_{z^2}(x,t;s)\}.$ The basic estimates we need to control the above terms come from \eqref{fd:bound} and the Sobolev embedding $H^2(\Omega) \hookrightarrow L^\infty(\Omega)$:
\begin{equation}
\label{decay3}
\|u_{z^1}(t;s)\|_{L^\infty} + \|u_{z^2}(t;s)\|_{L^\infty} \leq \Q(t_0), \qquad s \leq t \leq t_0.
\end{equation}
The straightforward estimates
$$
|\varphi'(u_{z^1}) \nabla_x (u_{z^1}-u_{z^2})|^2  
\leq c\big(1+\|u_{z^1} \|_{L^\infty}^{2}) |u_{z^1}-u_{z^2}|_1^2,     
$$
and
\begin{align*}
|\varphi''(\bar u) (u_{z^1}-u_{z^2}) \nabla_x u_{z^1}|^2  &\leq c\big(1+ \|u_{z^1}\|_{L^\infty}+\|u_{z^2}\|_{L^\infty}\big)\|u_{z^1}-u_{z^2}\|^2_{L^6} \|\nabla_x  u_{z^2} \|_{L^{\frac83}}^2   \\ & \leq c\big(1+ \|u_{z^1}\|_{L^\infty}+\|u_{z^2}\|_{L^\infty}\big) |u_{z^1}-u_{z^2}|^2_1 |   u_{z^2} |_{2}^2
\end{align*}
together with an application of  \eqref{fd:bound} and \eqref{decay3} lead us to the bound
\begin{align}
\label{decay4}
|  g(t;s)|_1^2 &\leq c|\nabla_x g(t;s)|^2 \leq \Q(t_0) |u_{z^1}(t;s)-u_{z^2}(t;s)|_1^2\\& \leq \Q(t_0) \exp\big((1+\cic{\nu}(t_0)) (t-s) \big)  \|  \bar z\|_{ X_{s}}^2,  \nonumber\end{align}
for all $s \leq t\leq t_0$, where we used \eqref{INTER-TILDE} in the last passage. 
Now, multiplying \eqref{SYS-TILDE} by $A \pt \bar k$
yields the differential inequality 
\begin{align}
\label{SYS-TILDE12} & \ddt \| (k(t;s),\pt k(t;s))\|^2_{X_t^1} + 2\omega(t) |A^{1/2} \pt k(t;s)|^2\\ \nonumber & \leq \mu'(t)|k(t;s)|_2^2 +  2 \l f+g,A \pt k 
\r \\ &\leq  \nonumber W \mu(t)|k(t;s)|_2^2 + \textstyle\frac{c}{\omega(t_0)}\big(|f(t;s)|_1^2 +c|g(t;s)|_1^2\big) + 2\omega(t) |A^{1/2} \pt k|^2,
\end{align}
for all $s \leq t\leq t_0$. The second passage involves the use of condition (M1) and of the decreasing monotonicity of $\omega$. We summarize  \eqref{decay3}, \eqref{decay4}, and \eqref{SYS-TILDE12} into
$$
\ddt \| (k(t;s),\pt k(t;s))\|^2_{X_t^1} \leq \Q(t_0)\left( \| (k(t;s),\pt k(t;s))\|^2_{X_t^1} +   \exp\big((1+\cic{\nu}(t_0)) (t-s) \big)  \|  \bar z\|_{ X_{s}}^2\right),
$$
for all $s \leq t\leq t_0$. An application of Gronwall's lemma on $(s,s+t_\star)$ and the observation that $(k(s;s),\pt k(s;s))=0$ yields 
\begin{equation}
\label{ottododici}
 \| (k(s+t_\star;s),\pt k(s+t_\star;s))\|^2_{X_t^1} \leq \Q(t_0) \exp\big((1+\cic{\nu}(t_0)) t_\star \big)  \|  \bar z\|_{ X_{s}}^2,
 \end{equation}
which is \eqref{smoooth} with   $\F(t_0,t_\star)$ precisely equal to the right-hand side of \eqref{ottododici}. The verification of (SP) is complete.

\subsection*{Acknowledgments.}  This work was partially
supported by the National Science Foundation under the grant  NSF-DMS-0906440, and by the Research Fund of Indiana University. The authors gratefully thank Noah Graham for sharing with them his insight in the physics of the problem and for bringing to their attention several relevant references, including \cite{FGGIRS}. The authors would like to thank the referee for his careful reading and insightful remarks, which helped us to improve the article.
 
\subsection*{Note added in proofs}{\it At the time when the galley proofs of this article were being corrected, one of the authors (RT) gave a lecture on the topic of this article and \cite{DDT}. Following are some useful remarks resulting from the lecture and the subsequent discussion. \vskip1.5mm \noindent 
$\cdot$ Equation \eqref{SYS-INTRO} contains two viscosity coefficients:
\begin{itemize}\item[$\cdot$]  $\mu(t)$, the viscosity coefficient of the spatial operator. In the original oscillon equation considered in \cite{DDT}, $\mu(t)=\e^{-2Ht}$.  The analysis of the asymptotic behavior is made difficult by both   singularities
$$\lim_{t \to + \infty}\e^{-2Ht} =0, \qquad \lim_{t \to -\infty}\e^{-2Ht} =+\infty.$$  The first difficulty
is overcome by using the pullback attractor framework (even the classical
framework without time dependent spaces would suffice),
where the absorbing family, and the attractor are  allowed to have size
depending explicitly on time and possibly going to $+\infty$ as $t\to + \infty$.  The second difficulty (which, in short, causes fixed balls of $H^1_0(\Omega) \times L^2(\Omega)$ to have huge ``physical'' energy for large negative time) is  circumvented with the introduction of the time
dependent spaces and the restriction of  the basin of
attraction to the pullback-bounded   families. 
\item[$\cdot$] $\omega(t)$, the damping coefficient of the evolution equation, which, classically, makes the dynamical system dissipative, as we have seen in the article.

\end{itemize}
 \vskip1.5mm \noindent 
$\cdot$ The quantity $u$ appearing in \eqref{SYS-INTRO}  represents a scalar field, usually taken to be the Higgs field.
 \vskip1.5mm \noindent 
$\cdot$ In the post-inflation scenario in which this equation is used, a(t) (see Remarks \ref{physcond} and \ref{graham-hyp}) is steadily increasing. We have introduced some mathematical generality in the time-dependent coefficients.
}

\end{document}